\documentclass[a4paper]{article}
\usepackage{color}
\usepackage[colorlinks,plainpages,hypertexnames=false,plainpages=false]{hyperref}
\hypersetup{urlcolor=blue, citecolor=blue, linkcolor=blue}

\usepackage{tikz}
\usetikzlibrary{3d}
\usetikzlibrary{calc}
\usepackage{pgfplots}
\pgfplotsset{compat=1.11}
\usepgfplotslibrary{fillbetween}
\usetikzlibrary{intersections}

\usepackage[margin=1in]{geometry}
\usepackage{url}
\usepackage{verbatim}
\usepackage{enumitem}
\usepackage{textcomp}
\usepackage{pgfplots}
\usepackage{dsfont}
\usepackage{tabularx}

\usepackage{eqparbox}
\usepackage{makebox}
\newcommand\textoverset[3][]{\mathrel{\overset{\scriptsize{\eqmakebox[#1]{#2}}}{#3}}} %

\makeatletter %
\newcommand{\raisemath}[1]{\mathpalette{\raisem@th{#1}}}
\newcommand{\raisem@th}[3]{\raisebox{#1}{$#2#3$}}
\makeatother

\newcommand\subsmash[1]{{\smash[b]{#1}}}

\usepackage{amsmath,amsthm,amssymb,bbold,mathtools}
\usepackage[capitalize]{cleveref}

\newcommand{\A}{\mathcal{A}}

\newcommand{\RR}{\mathbb{R}}

\DeclareMathOperator{\Trop}{Trop}

\DeclareMathOperator{\conv}{conv}
\DeclareMathOperator{\Supp}{Supp}

\newcommand{\suchthat}{\;\ifnum\currentgrouptype=16 \middle\fi|\;}

\title{Sharp bounds for the number of regions of maxout networks\\ and vertices of Minkowski sums} %

\author{Guido Mont\'ufar\thanks{Department of Mathematics and Department of Statistics, UCLA, Los Angeles, CA 90095, USA; Max Planck Institute for Mathematics in the Sciences, 04103 Leipzig, Germany ({montufar@math.ucla.edu}).}
\and Yue Ren\thanks{Department of Mathematics, Swansea University, Swansea SA1 8EN,
United Kingdom ({yue.ren@swansea.ac.uk}).}
\and Leon Zhang\thanks{Department of Mathematics, UC Berkeley, Berkeley, CA 94720, USA ({leonyz@berkeley.edu}).}}

\newtheorem{theorem}{Theorem}[section]

\newtheorem{corollary}[theorem]{Corollary}
\newtheorem{lemma}[theorem]{Lemma}
\newtheorem{proposition}[theorem]{Proposition}

\newtheorem{definition}[theorem]{Definition}
\newtheorem{example}[theorem]{Example}

\let\OLDthebibliography\thebibliography
\renewcommand\thebibliography[1]{
  \OLDthebibliography{#1}
  \setlength{\parskip}{0pt}
  \setlength{\itemsep}{2mm}
}

\begin{document}

\maketitle

\begin{abstract}
We present results on the number of linear regions of the functions that can be represented by artificial feedforward neural networks with maxout units.
A rank-$k$ maxout unit is a function computing the maximum of $k$ linear functions.
For networks with a single layer of maxout units, the linear regions correspond to the upper vertices of a Minkowski sum of polytopes. We obtain face counting formulas in terms of the intersection posets of tropical hypersurfaces or the number of upper faces of partial Minkowski sums, along with explicit sharp upper bounds for the number of regions for any input dimension, any number of units, and any ranks, in the cases with and without biases. Based on these results we also obtain asymptotically sharp upper bounds for networks with multiple layers.
\end{abstract}

\section{Introduction}

Artificial feedforward neural networks are parametric sets of functions given as fixed compositions of units, i.e.\ elementary functions consisting of a parametrized affine map followed by a fixed %
activation function. They are extremely useful as sets of hypothesis functions in contemporary machine learning applications.
We are interested in the geometry and combinatorics of the functions represented by networks with maxout units, which have an activation function of the form $\mathbb{R}^k\to \mathbb{R};\; (z_1,\ldots,z_k)\mapsto \max\{z_1,\ldots,z_k\}$.
Maxout units were proposed in \cite{pmlr-v28-goodfellow13} generalizing the popular rectified linear units (ReLUs)~\cite{pmlr-v15-glorot11a}, which have activation function $\mathbb{R}\to \mathbb{R};\; z\mapsto \max\{0,z\}$.
The corresponding functions are piecewise (affine) linear and induce subdivisions of the input space into linear regions.
We will be concerned with estimating the maximum number of linear regions of the functions that can be represented by maxout networks with given architectures.

The analysis of neural networks with piecewise linear activation functions based on the number of linear regions of the represented functions was proposed in \cite{pascanu2013number,NIPS2014_5422}, showing that deep networks %
can represent functions which have many more linear regions than any of the functions that can be represented by shallow networks with the same number of units or the same number of parameters.  These kind of results illustrate the differences in representational power and possible benefits of different network architectures.
Works in this direction include
\cite{6704758,
telgarsky2015representation,
pmlr-v49-eldan16,
pmlr-v49-telgarsky16,
YAROTSKY2017103,
10.5555/3298483.3298577,
pmlr-v70-raghu17a} and earlier works for Boolean circuits and sum-product networks \cite{Hastad86,Hastad91,NIPS2011_8e6b42f1}.
The number of linear regions of the functions represented by networks with piecewise linear activations has sparked substantial interest in the study of neural networks, with works including
\cite{NIPS2014_5422,
pmlr-v49-telgarsky16,
montufar2017notes,
arora2018understanding,
DBLP:conf/icml/SerraTR18,
8756157}.
Recent works have %
explored approaches based on tropical geometry \cite{pmlr-v80-zhang18i,
Charisopoulos2018ATA,
alfarra2020on} and power diagram subdivisions~\cite{NIPS2019_9712}, %
while others have studied %
the expectated %
number of linear regions for typical choices of the parameters in the case of ReLU networks \cite{pmlr-v97-hanin19a,
NIPS2019_8328},  empirical enumeration~\cite{Serra_Ramalingam_2020},
and the relations between linear regions and the behavior of algorithms that are used to select the parameters of neural networks based on data, such as speed of convergence and implicit biases of gradient descent~\cite{Steinwart2019ASL,
Zhang2020Empirical,
86441}.
ReLU networks have been studied in much more detail than maxout networks.
And while ReLUs %
are currently more popular in applications, maxout networks are an interesting generalization %
that enjoy similar benefits (e.g.\ linear operation, no saturation) but without some of the possible %
drawbacks (e.g.\ dying neurons). In this work we seek to advance %
the theory for maxout networks, particularly in regards to their representational power, whereupon we develop important connections to topics in combinatorial geometry and tropical geometry.
The nonlinear locus of a ReLU $x\mapsto \max\{0,z(x)\}$ with a generic affine function $z$ is a hyperplane. Hence, linear regions of functions $x\mapsto [\max\{0,z_1(x)\},\ldots, \max\{0,z_m(x)\}]^\top$ represented by a layer of $m$ ReLUs are described by hyperplane arrangements.
Hyperplane arrangements have been investigated since the 19th century \cite{Steiner1826}.
In particular, %
Buck \cite{10.2307/2303424} showed that the number of regions and bounded regions that can be obtained by slicing $n$-dimensional Euclidean space with $m$ hyperplanes is $\sum_{j=0}^n{m\choose j}$ %
and ${m-1\choose n}$, respectively.
Moreover, a celebrated result by Zaslavsky~\cite{zaslavsky1975facing} gives a formula for the number of faces and bounded faces of hyperplane arrangements based on the poset of intersections.
For a discussion of these results and other properties of hyperplane arrangements see
 \cite{Stanley04anintroduction}.

For maxout units one obtains a more general type of arrangement.
Concretely, a rank-$k$ maxout unit
$x\mapsto \max\{z_1(x),\ldots, z_k(x)\}$ with generic affine functions $z_1,\ldots, z_k$ has a nonlinear locus of the form $\{x\in\mathbb{R}^n \colon z_i(x)=z_j(x)=\max\{z_1(x),\ldots, z_k(x)\}\allowbreak \text{ for some $i\neq j$}\}$. In tropical geometry this is known as a tropical hypersurface~\cite{MaclaganSturmfels15}.
Hence the linear regions of the functions $x\mapsto [\max\{z_{11}(x),\ldots, z_{1k_1}(x)\},\ldots,\max\{z_{m1}(x),\ldots, z_{mk_m}(x)\}]^\top$ represented by a layer of $m$ maxout units of ranks $k_1,\ldots,k_m$ are described by arrangements of tropical hypersurfaces. We will also refer to these as maxout arrangements.
The properties of such arrangements are not as well understood, except in special cases, such as tropical hyperplane arrangements \cite{Speyer08,Federico,JaggiKatzWagner}, which correspond to networks with restricted parameters, namely maxout networks whose affine maps are coordinate projections plus constants, $z_{ij}(x)=x_j + b_{ij}$.

In order to obtain counting formulas and bounds for maxout networks, we will exploit a correspondence between the regions of maxout arrangements and the upper vertices of Minkowski sums of polytopes (Proposition~\ref{prop:uppervertreg}).
In the special case of hyperplane arrangements (rank-$2$ maxout units), these reduce to Minkowski sums of line segments, called zonotopes \cite{z-lop-95}.
Minkowski sums of polytopes are of relevance in numerous topics, including computational commutative algebra, collision detection, robot motion planning, computer-aided design, and have been the subject of an intensive research program over the years.
In particular, the work of Gritzmann and Sturmfels~\cite{GritzmannSturmfels} %
showed that for sums of polytopes %
with at most $r$ total nonparallel edges, %
the maximum number of faces is attained by sums of $r$ line segments in general position. %
Tight expressions for the maximum number of faces of Minkowski sums of two and three full-dimensional polytopes were derived in \cite{10.1007/s00454-015-9726-6,10.1145/2462356.2462368}. %
Relevant to our discussion, Weibel \cite{Weibel12}
obtained an expression for the number of faces of large Minkowski sums of full-dimensional polytopes in terms of the number of faces of small subsums, and tight upper bounds for the total number of vertices.
Obtaining similar results for sums of polytopes of arbitrary dimensions requires significantly more complex arguments.
A full solution to the so-called upper bound problem for Minkowski sums (UBPM) was obtained by Adiprasito and Sanyal \cite{KarimRaman}, giving %
tight upper bounds (in non-closed form) for the number of faces of any dimension, of Minkowski sums of any polytopes with given numbers of vertices.

We shall take Zaslavsky's perspective (Theorem~\ref{thm:posetcounting}) to extend Weibel's result to the case of sums of polytopes of arbitrary dimensions (Theorem~\ref{thm:central-arbdims}), including the treatment of upper vertices (Theorem~\ref{thm:facessimple}) and bounds on the number of strict lower vertices (Theorem~\ref{thm:lowerbound_strict_lower}). Combining these with an implication of Adiprasito-Sanyal's result (Proposition~\ref{thm:Mneighborly_upper_bound}), we obtain  explicit tight upper bounds for the total number of vertices and for the number of upper vertices.
Our results for Minkowski sums of polytopes translate to tight upper bounds on the number of linear regions of the functions represented by shallow maxout networks without and with biases (Theorem~\ref{thm:main-result}). These are the first tight results for maxout networks (except for the rank-$2$ case), closing significant gaps between the upper and lower bounds from previous works \cite{NIPS2014_5422}.
Based on these results, we also derive results for deep maxout networks (Theorem~\ref{thm:deep-result}) improving previous lower and upper bounds \cite{NIPS2014_5422,DBLP:conf/icml/SerraTR18}.

\smallskip

This article is organized as follows.
In Section~\ref{sec:perspectives} we provide definitions and different perspectives on maxout networks and their linear regions. %
In Section~\ref{sec:main} we present our main results on the maximum number of linear regions of maxout networks. %
The main analysis is conducted in Sections~\ref{sec:Weibel}, \ref{sec:Zaslavsky}, and \ref{sec:Weibel-Zaslavsky}.
In Section~\ref{sec:Weibel} we present a modification of a result by Weibel to count the upper faces of Minkowski sums.
In Section~\ref{sec:Zaslavsky} we present a modification of a result by Zaslavski to count the faces of maxout arrangements.
In Section~\ref{sec:Weibel-Zaslavsky} we combine these two approaches to obtain a generalization of Weibel's result to sums of polytopes of arbitrary dimensions.
In Section~\ref{sec:discussion} we offer a discussion and outlook.

\section{Definitions and different perspectives}
\label{sec:perspectives}

\subsection{Maxout networks}

We consider standard feedforward fully connected maxout networks with no skip connections, called maxout networks for short.
Maxout networks were introduced in \cite{pmlr-v28-goodfellow13} as a generalization of ReLU networks.
\begin{definition}[Maxout networks]\
\begin{enumerate}[leftmargin=*]
\item
Let $k,n\in\mathbb{N}$. A rank-$k$ maxout unit with $n$ inputs is a parametric function
$$
\mathbb{R}^n\to\mathbb{R}; \quad
x\mapsto \max\{ \langle A_{1}, x\rangle  + b_{1} ,\ldots, \langle A_{k}, x\rangle  + b_{k} \},
$$
parametrized by
$\theta = (A_r,b_r)_{r=1}^k$, $A_r\in \mathbb{R}^n$, $b_r\in \mathbb{R}$, $r=1,\ldots, k$.
Each affine function is called a pre-activation feature.
The parameters $A_r$ and $b_r$ are called %
weights and biases. %
\item Let $m\in\mathbb{N}$ and $k_1,\ldots, k_m\in\mathbb{N}$.
A layer with $n$ inputs and $m$ maxout output units of ranks $k_1,\ldots, k_m$ is
a parametric function
$\mathbb{R}^{n}\to\mathbb{R}^{m}$,
whose $j$th output coordinate is a maxout unit of rank $k_j$, $j=1,\ldots, m$.
A layer of maxout units is also called a shallow maxout network.
\item
Let $L\in\mathbb{N}$ and $n_0,n_1,\ldots, n_L\in\mathbb{N}$.
A maxout network with $n_0$ inputs and $L$ layers of widths $n_1,\ldots, n_L$ is %
a parametric function
$\mathbb{R}^{n_0}\to \mathbb{R}^{n_L}$ of the form
$f_L \circ \cdots \circ f_1$,
where $f_l$ is a function represented by a layer with $n_{l-1}$ inputs and $n_l$ maxout output units of given ranks, $l=1,\ldots, L$. A network with multiple layers is called a deep network.

\item
The architecture of a maxout network as described above is determined by the number of inputs $n_0$, the number of layers $L$, the number of units per layer $n_1,\ldots, n_L$, and the ranks of the maxout units in each layer $k_{l,1},\ldots, k_{l,n_l}$, $l=1,\ldots, L$.
The parameter of the network, which we will denote by $\theta$, is the collection of weights and biases of all units.
For a given architecture, we denote $\mathcal{N}$ the set of functions that can be represented by the network for all possible choices of the parameter.
\end{enumerate}

\end{definition}

We will be concerned mostly with the analysis of shallow networks, from which we will also derive results for deep networks.
We will also present results for networks without biases, in which case the affine functions $\langle A_r,x\rangle + b_r$ reduce to linear functions $\langle A_r,x\rangle$.
Notice that each function represented by a maxout network is a composition of continuous piecewise (affine) linear functions and is itself a continuous piecewise linear function.
When there is no risk of confusion we refer to affine linear functions simply as linear functions.

\begin{definition}[Linear regions]
Let $f$ %
be a continuous piecewise linear function %
with $n_0$ inputs.
The nonlinear locus of %
$f$ is the set $V(f)\subseteq\mathbb{R}^{n_0}$ of input points $x$ at which $\nabla_x f$ is discontinuous.
A linear region of $f$ is a maximal connected component of $\mathbb{R}^{n_0}\setminus V$.
The number of linear regions of $f$ is denoted $N(f)$.
The maximum of the number of linear regions among all functions $f$ that can be expressed by a network $\mathcal{N}$ is denoted  $N(\mathcal{N}) = \max_{f\in\mathcal{N}}N(f)$.
\end{definition}

\subsection{Tropical hypersurfaces}

Maxout units can be regarded as tropical polynomials.
We give a brief description of these notions. For more details, please see \cite[Chapter 1]{JoswigBook}.

\begin{definition}
\label{defn:tropLaurent}
Given two real numbers $a$ and $b$, their \emph{tropical sum} is $a\oplus b = \max(a,b)$ and their \emph{tropical product} is $a\odot b = a + b$. A \emph{tropical (exponential) polynomial} is a function %
\[
f\colon\RR^{n_0}\to\RR; \quad f(x) = c_1 \odot x^{\odot\alpha_1} \oplus \dots \oplus c_k \odot x^{\odot\alpha_k},
\]
where $c_1,\dots, c_k\in \RR$, $\alpha_1,\dots, \alpha_k \in$ $\RR^{n_0}$, and $x^{\odot\alpha_1} =\alpha_1x_1+\dots \alpha_{n_0}x_{n_0}$. We refer to the $c_i \odot x^{\odot\alpha_i}$ as \emph{tropical monomials} and call $f$ a \emph{tropical $k$-nomial} if it is the sum of $k$ distinct monomials.
\end{definition}

Classically, polynomials (tropical or non-tropical) only have non-negative integer exponents.
However, this restriction is not needed in our %
discussion.
A rank-$k$ maxout unit is equivalent to a tropical $k$-nomial.

\begin{definition}
The tropical hypersurface of a tropical polynomial $f:\RR^{n_0}\to \RR$ is
\[\Trop(f):= \{x\in \RR^{n_0}: c_i x^{\alpha_i}=c_j x^{\alpha_j} = f(x)\text{ for } i\ne j \text{ two distinct monomials}\}.\]
\end{definition}
The complement $\RR^{n_0}\backslash \Trop(f)$ is a union of convex polyhedral cells on which the function $f$ is linear.
In particular, the nonlinear locus of a maxout unit is a tropical hypersurface.

From the tropical perspective, the goal of the paper is to answer a question in tropical combinatorics, namely to bound the number of regions of an arrangement of tropical hypersurfaces. Similar questions on the combinatorics of tropical hypersurface arrangements have been studied before. However, they often focus on polynomials of bounded degree, e.g.\ degree $1$ \cite{Federico,JaggiKatzWagner} or degree $d$ \cite{Joswig_2017}, rather than a bounded number of monomials. The bounds are important for the complexity of many algorithms in tropical geometry, which often rely on a enumeration of cells in a tropical arrangement or tropical variety.

\subsection{Convex conjugates and Newton polytopes}

We relate the regions of a maxout network and the vertices of certain polytopes.
The procedure has been explained in~\cite{pmlr-v80-zhang18i} from a tropical geometry perspective. %
We give a brief description that relies only on convex duality.
For an introduction to convex analysis see \cite{10.2307/j.ctt14bs1ff}.

Any continuous piecewise linear function can be expressed as the difference of two convex piecewise linear functions; see~\cite{1333237}.
Hence any function $f\colon \mathbb{R}^n\to\mathbb{R}^o$ expressed by a maxout network can be written, for each output coordinate $i=1,\ldots, o$, as the difference $f_i=g_i-h_i$ of two piecewise linear convex functions $g_i$ and $h_i$.
Given such a decomposition, we define a surrogate function $\bar f = g+h\colon \mathbb{R}^n\to\mathbb{R}$, where $g=\sum_{i=1}^o g_i$ and $h=\sum_{i=1}^o h_i$ are convex. %
This is a piecewise linear convex function and hence it can be written as
\begin{equation}
\bar f(x) = \max_{j=1,\ldots, M}  \{\langle a_j, x\rangle + b_j\}, \quad x\in \mathbb{R}^n,
\label{eq:maxoutnetworkconvexfunction}
\end{equation}
for a finite collection of coefficients $a_j\in \mathbb{R}^n$, $b_j\in \mathbb{R}$, $j=1,\ldots, M$. We will discuss the coefficients in more detail further below.

Now, any linear region of $f$ is a union of linear regions of $\bar f$, or in other words $f$ is linear on all linear regions of $\bar f$:
Let $R$ be a linear region of $\bar f$. Assume w.l.o.g. that $\bar f$ is constant $0$ on $R$, i.e., $\bar f|_R=g|_R+h|_R=0$. This implies that $g|_R=-h|_R$, which means that $g$ and $h$ are both convex and concave on $R$. Therefore, they must be linear on $R$ and so must be $f$.

Moreover, any two distinct linear regions $R,Q\subseteq\mathbb{R}^n$ of $\bar f$ %
are also distinct linear regions of $f$ %
unless $\nabla g|_{R} -\nabla g|_Q = \nabla h|_{R}- \nabla h|_{Q}$, a tie which is broken, for instance, whenever $g$ is scaled independently of $h$.
One can show that for generic choices of the network parameters $f$ and $\bar f$ actually have the same number of linear regions.

Each linear region of $\bar f$ corresponds to a (full dimensional) neighborhood of inputs $x$ over which one of the affine functions $\langle a_j, x\rangle + b_j$ attains the maximum. %
A representation of $\bar f$ as in \eqref{eq:maxoutnetworkconvexfunction} may involve many redundant affine functions.
One way to characterize the affine functions which attain the maximum over a neighborhood of the input space is as follows.
Consider the convex conjugate of $\bar f$, which is the convex piecewise linear function
$$
{\bar f}^\ast(x^\ast) = \sup_{x\in\mathbb{R}^n} \{ \langle x,x^\ast\rangle -\bar f(x)\}, \quad x^\ast\in\mathbb{R}^n.
$$
If $\bar f(x) = \langle a, x\rangle +b$
for some $x\in\mathbb{R}^n$ and $(a,b)\in \mathbb{R}^{n+1}$, then $(x^\ast,{\bar f}^\ast(x^\ast)) = (a,-b)$.
We conclude that the graph of ${\bar f}^\ast$ is the convex envelope of the points $\{(a_j,-b_j)\}\subseteq\mathbb{R}^{n+1}$.
The vertices of this envelope correspond to the affine functions $\langle a_j,x\rangle +b_j$ which attain the maximum over a neighborhood of the input space.
An equivalent way of expressing this is as follows.

\begin{definition}[Lifted Newton polytope]
\label{def:Newton-polytope}
For a function %
$\bar f(x) = \max_{j=1,\ldots, M}  \{\langle a_j, x\rangle + b_j\}$, $x\in\mathbb{R}^n$, define its
Newton polytope as $\operatorname{conv}\{a_j\colon j=1,\ldots, M\}\subseteq\mathbb{R}^{n}$, and its
lifted Newton polytope as
$P_{\bar f} =\operatorname{conv}\{(a_j,b_j)\colon j=1,\ldots, M\}\subseteq\mathbb{R}^{n+1}$.
\end{definition}

The upper vertices of a polytope $P\subseteq\mathbb{R}^{n+1}$ are the vertices $p$ which are `visible from above', meaning that their normal cones $\{r\in\mathbb{R}^{n+1}\colon \langle r, p-q\rangle> 0 \text{ for all } q\in P \setminus p\}$ intersect the upper halfspace $\mathbb{R}^n\times \mathbb{R}_{>0}$.

\begin{proposition}[{\cite[Theorem 1.13]{JoswigBook}}]
\label{prop:ftt}
The linear regions of %
$\bar f(x) = \max_{j\in\{1,\ldots, M\}}\{\langle a_j , x\rangle + b_j\}$, $x\in \mathbb{R}^n$, correspond to the upper vertices of its lifted Newton polytope $P_{\bar f}$. %
\end{proposition}

The duality between faces of the nonlinear locus and faces of the lifted Newton polytope goes beyond what is described in Proposition~\ref{prop:ftt}, and it is fundamental for studying the combinatorics of tropical hypersurfaces. %

\subsection{Minkowski sums}
The lifted Newton polytopes for single layer networks have a description as Minkowski sums.
Recall that the Minkowski sum of two sets $A$ and $B$ is defined as $A+B = \{a+b\colon a\in A, b\in B\}$.
For a single layer of $m$ maxout units of ranks $k_1,\ldots, k_m$, %
$$
\bar f(x)
\!=\! \sum_{j\in[m]} \max_{r_j\in[k_j]}\{\langle w_{j,r_j}, x\rangle + b_{j,r_j}\}
\!=\! \max_{r\in [k_1]\times \cdots\times[k_m]} \Big\{ \sum_{j\in[m]} \langle w_{j,r_j}, x\rangle + b_{j,r_j}\Big\}
\!=\! \max_{r} \{\langle w_r  , x\rangle + b_r\},
$$
where $[k]:=\{1,\ldots, k\}$, and $w_r = \sum_{j\in[m]} w_{j,r_j}$, $b_r = \sum_{j\in[m]} b_{j,r_j}$, for $r\in [k_1]\times \cdots\times[k_m]$.
To see this one may use that the distributive law holds for tropical addition and multiplication %
 \cite[Section~1.1]{MaclaganSturmfels15}.
Notice that the resulting set of coefficients is the Minkowski sum of the sets of coefficients of the individual units,
$\{(w_r,b_r)\colon r\in [k_1]\times\dots\times[k_m]\} = \sum_{j\in[m]} \{(w_{j,r_j},b_{j,r_j})\colon r_j\in[k_j]\}$.
Now consider the lifted Newton polytope of the %
layer and the polytopes of the individual units,
$
P_{[m]} = \conv\left\{(w_r,b_r)\colon r\in [k_1]\times\dots\times[k_m]\right\}$,
$P_j= \conv \left\{(w_{j,r_j},b_{j,r_j})\colon r_j\in[k_j]\right\}, \; j\in[m]$.
Since the Minkowski sum of convex sets is convex, we have
\begin{equation*}
P_{[m]} = \sum_{j\in[m]} P_{j}.
\end{equation*}
In turn, the polytope for a layer is obtained by taking the Minkowski sum of the polytopes of the individual units.
This is all we will need in our discussion.
We note that, following the arguments of \cite{pmlr-v80-zhang18i}, one can also describe how the polytopes corresponding to layers are combined to obtain a polytope for a deep network. That work focused on ReLU networks, but the same arguments extend naturally to maxout networks.

We collect a few observations in the next proposition.
For maxout units of ranks $k_1,\ldots, k_m$ the polytopes $P_1,\ldots, P_m$ are arbitrary convex hulls of, respectively, $k_1,\ldots, k_m$ points in $\mathbb{R}^{n+1}$.
For units without biases, the last coordinate of the coefficients is always zero, so that the polytopes are in $\mathbb{R}^n\times\{0\}$ and the upper vertices are simply the vertices.

\begin{proposition}
\label{prop:uppervertreg}
The %
linear regions of a function represented by a layer with $n$ inputs and maxout units of ranks $k_1,\ldots, k_m$
correspond to the upper vertices of a Minkowski sum of polytopes which are convex hulls of $k_1,\ldots, k_m$ points in $\mathbb{R}^{n+1}$.
For a layer without biases, the %
linear regions correspond to the %
vertices of a Minkowski sum of polytopes which are convex hulls of $k_1,\ldots, k_m$ points in $\mathbb{R}^n$.
\end{proposition}

The difference between counting the vertices vs %
the upper vertices of Minkowski sums of polytopes is analogous to the difference between counting the regions of a central arrangement (without biases) vs %
counting the regions of a non-central arrangement (with biases). %
The nonlinear locus $V(f)$ of a function $f$ represented by a layer of maxout units without biases is the normal fan of the Newton polytope.
The nonlinear locus of a function with biases is the intersection of the normal fan of the lifted Newton polytope in $\mathbb{R}^{n+1}$ and the affine space $\mathbb{R}^n\times \{1\}$.
The hardest part in our proofs will be to upper bound the number of regions of non-central arrangements.

A result of particular importance in our analysis is the Upper Bound Theorem for Minkowski Sums by Adiprasito and Sanyal \cite[Theorem 6.11]{KarimRaman}, which shows that the maximum number %
of $s$-dimensional faces %
of Minkowski sums of polytopes with given numbers of vertices %
is attained by sums of %
so-called Minkowski neighborly families.
From their result we derive the following proposition, which we will later use to obtain an explicit form of the upper bound %
for vertices. %
It will also be an ingredient in our upper bound for upper vertices.
For a 
polytope $P$, let $f_s(P)$ denote the number of $s$-dimensional faces of $P$.

\begin{proposition}
\label{thm:Mneighborly_upper_bound}
Let $0\leq s\leq n$. %
If a Minkowski sum of polytopes $P = P_1+\cdots+P_m\subseteq\mathbb{R}^{n+1}$ has the maximum number of $s$-faces among all sums with given $f_0(P_1),\ldots, f_0(P_m)$, then $f_0(\sum_{i\in S}P_i)=\prod_{i\in S}f_0(P_i)$ for all $S\subseteq [m]$, $|S|\leq n$. %
\end{proposition}
Intuitively, this proposition states that if a sum of polytopes in $\mathbb{R}^n$ with given vertex counts reaches the largest possible number of vertices, then each partial sum of at most $n$ of the polytopes reaches a trivial upper bound on the number of vertices.

\begin{proof}
We explain how to derive the claim based on %
\cite{KarimRaman}.
We refer the reader to that paper for details on Minkowski neighborly families, Cayley polytopes, (relative) Cayley complexes, $h$-vectors, and the corresponding Dehn-Sommerville relations.
That paper shows that a Minkowski sum $P_1+\cdots+ P_m$ of polytopes $P_i\subseteq\mathbb{R}^d$ attains the maximum number of $k$-faces, $0\leq k\leq d$,  %
if the family $(P_1,\ldots, P_m)$ is Minkowski neighborly.
(Following their notation, here we use $k$ for the dimension of the faces). %
Of particular interest is the classification of cases maximizing the number of faces of a particular dimension $k_0$.
Following \cite[Theorem 6.11]{KarimRaman}, for a given $k_0$, a Minkowski sum $P_1+\cdots+P_m$ attains the maximum number of $k_0$-faces if and only if the $h$-vector of its relative Cayley complex attains maximum value at all entries $h_{k+m-1}$ with $k\leq k_0+1$.
By the Dehn-Sommerville relations, the entries with $k+m-1>\lfloor\frac{d+m-1}{2}\rfloor$ are determined as positive linear functions of the entries with $k+m-1\leq \lfloor\frac{d+m-1}{2}\rfloor$ of the $h$-vectors for sub-families $(P_i)_{i\in U}$, $U\subseteq [m]$.
Hence we only need to verify that the latter entries are maximal.
By \cite[Theorem~6.11(2a)]{KarimRaman}, this is equivalent to verifying that for any $k'+m'-1 \leq \lfloor\frac{d+m'-1}{2}\rfloor$ and $U\subseteq[m]$, $|U|=m'$,
the following holds.
For all $S\subseteq U$, all nonfaces of the Cayley polytope $T_S$ of cardinality $k'+|S|-1$ are supported in some vertex set $V(T_R)$ with $R\subsetneq S$.
One can rewrite the cases as $k'\leq \lfloor\frac{d+m'-1}{2}\rfloor - (m'-1)$ and nonfaces of cardinality $\leq  \lfloor\frac{d-(m'-1)}{2}\rfloor  + |S|-1$.
For $m'\geq d$, the condition is trivially satisfied.
Hence we only need to check cases with $m'=d-r$, $r\geq 1$, and nonfaces of cardinality $\leq \lfloor\frac{r+1}{2}\rfloor + |S|-1$, where $|S|\leq d-r$, $r\geq1$.
This means that for any $S\subseteq[m]$ of cardinality $|S|\leq n$, any selection of one vertex from each $P_i$, $i\in S$, results in a vertex of the polytope $P_S=\sum_{i\in S} P_i$, and hence that  $P_S$ attains the trivial upper bound on the number of vertices, $f_0(P_S) = \prod_{i\in S} f_0(P_i)$.
\end{proof}

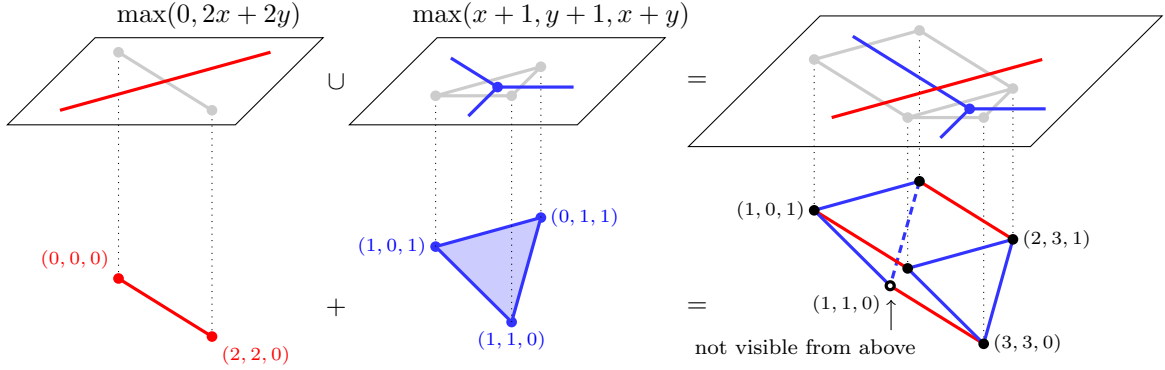
\begin{figure}
  \centering
  \begin{tikzpicture}
    \node (edge) at (-1,0)
    {%
      \begin{tikzpicture}[font=\scriptsize]
        \useasboundingbox (-2.5,0,-2.5) -- (0.5,0,-2.5) -- (0.5,0,0.5) -- (-2.5,0,0.5) -- cycle;
        \draw[dotted] (0,0,0) -- ++(0,3)
        (-2,0,-2) -- ++(0,3);
        \fill[red] (0,0,0) circle (2pt)
        (-2,0,-2) circle (2pt);
        \draw[red, very thick] (0,0,0) -- (-2,0,-2);
        \node[anchor=south east,red] at (-2,0,-2) {$(0,0,0)$};
        \node[anchor=north west,red] at (0,0,0) {$(2,2,0)$};      \end{tikzpicture}
    };
    \node (edgeDual) at (-1,3)
    {%
      \begin{tikzpicture}
        \useasboundingbox (-2.5,0,-2.5) -- (0.5,0,-2.5) -- (0.5,0,0.5) -- (-2.5,0,0.5) -- cycle;
        \draw (-2.5,0,-2.5) --  node[anchor=south] {$\max(0,2x+2y)$}  (0.5,0,-2.5) -- (0.5,0,0.5) -- (-2.5,0,0.5) -- cycle;
        \fill[black!20] (0,0,0) circle (2pt)
        (-2,0,-2) circle (2pt);
        \draw[black!20,very thick] (0,0,0) -- (-2,0,-2);
        \draw[very thick,red] (-1,0,-1) -- ++(1,0,-1)
        (-1,0,-1) -- ++(-1,0,1);
      \end{tikzpicture}
    };
    \node (triangle) at (3.5,0)
    {%
      \begin{tikzpicture}[font=\scriptsize]
        \useasboundingbox (-0.75,0,-1) -- (2.25,0,-1) -- (2.25,0,2) -- (-0.75,0,2) -- cycle;
        \fill[blue!80] (1,1,0) circle (2pt)
        (0,1,1) circle (2pt)
        (1,0,1) circle (2pt);
        \draw[blue!80,fill=blue!20,very thick] (1,1,0) -- (0,1,1) -- (1,0,1) -- cycle;
        \draw[dotted] (1,1,0) -- ++(0,2)
        (0,1,1) -- ++(0,2)
        (1,0,1) -- ++(0,3);
        \node[anchor=east,blue] at (0,1,1) {$(1,0,1)$};
        \node[anchor=west,blue] at (1,1,0) {$(0,1,1)$};
        \node[anchor=north,blue] at (1,0,1) {$(1,1,0)$};
      \end{tikzpicture}
    };
    \node (triangleDual) at (3.5,3)
    {%
      \begin{tikzpicture}
        \useasboundingbox (-0.75,0,-1) -- (2.25,0,-1) -- (2.25,0,2) -- (-0.75,0,2) -- cycle;
        \draw (-0.75,0,-1) -- node[anchor=south] {$\max(x+1,y+1,x+y)$}  (2.25,0,-1) -- (2.25,0,2) -- (-0.75,0,2) -- cycle;
        \fill[black!20] (1,0,0) circle (2pt)
        (0,0,1) circle (2pt)
        (1,0,1) circle (2pt);
        \draw[black!20,very thick, line join=bevel %
        ] (1,0,0) -- (0,0,1) -- (1,0,1) -- cycle;
        \coordinate (o) at (0.7,0,0.7);
        \fill[blue!80] (o) circle (2pt);
        \draw[very thick,blue!80] (o) -- ++(1,0,0)
        (o) -- ++(0,0,1)
        (o) -- ++(-1,0,-1);
      \end{tikzpicture}
    };
    \node (minkowski) at (9,0)
    {%
      \begin{tikzpicture}[font=\scriptsize]
        \useasboundingbox (-2.5,0,-2.5) -- (2,0,-2.5) -- (2,0,2) -- (-2.5,0,2) -- cycle;
        \node[anchor=west] at (1,1,0) {$(2,3,1)$};
        \node[anchor=west] at (1,0,1) {$(3,3,0)$};
        \node[anchor=east] at (-2,1,-1) {$(1,0,1)$};
        \node[anchor=north east] at (-1,0,-1) {$(1,1,0)$};
        \draw[dotted] (-2,1,-1) -- ++(0,2)
        (-1,1,-2) -- ++(0,2)
        (0,1,1) -- ++(0,2)
        (1,1,0) -- ++(0,2)
        (1,0,1) -- ++(0,3);
        \draw[blue!80,very thick, densely dashed]
        (-1,1,-2) -- (-1,0,-1);
        \draw[red,very thick]
        (-2,1,-1) -- (0,1,1)
        (-1,0,-1) -- (1,0,1)
        (-1,1,-2) -- (1,1,0);
        \draw[blue!80,very thick]
        (1,1,0) -- (0,1,1) -- (1,0,1) -- cycle
        (-2,1,-1) -- (-1,0,-1)
        (-1,1,-2) -- (-2,1,-1);
        \fill (1,1,0) circle (2pt)
        (0,1,1) circle (2pt)
        (1,0,1) circle (2pt)
        (-1,1,-2) circle (2pt)
        (-2,1,-1) circle (2pt);
        \fill[draw=black,fill=white,very thick]
        (-1,0,-1) circle (1.5pt);
      \end{tikzpicture}
    };
    \node (minkowskiDual) at (9,3)
    {%
      \begin{tikzpicture}
        \draw (-2.5,0,-2.5) -- (2,0,-2.5) -- (2,0,2) -- (-2.5,0,2) -- cycle;
        \fill[black!20] (1,0,0) circle (2pt)
        (0,0,1) circle (2pt)
        (1,0,1) circle (2pt)
        (-1,0,-2) circle (2pt)
        (-2,0,-1) circle (2pt);
        \draw[black!20,very thick,line join=bevel] (1,0,0) -- (0,0,1) -- (1,0,1) -- cycle
        (1,0,0) -- (-1,0,-2) -- (-2,0,-1) -- (0,0,1);
        \coordinate (o) at (0.7,0,0.7);
        \fill[blue!80] (o) circle (2pt);
        \draw[very thick,blue!80] (o) -- ++(1,0,0)
        (o) -- ++(0,0,1)
        (o) -- ++(-2.5,0,-2.5);
        \draw[very thick,red] (0,0,0) -- ++(1,0,-1)
        (0,0,0) -- ++(-1,0,1);
      \end{tikzpicture}
    };
    \node at (1.25,0) {$+$};
    \node at (1.25,3) {$\cup$};
    \node at (6,0) {$=$};
    \node at (6,3) {$=$};
    \node[font=\footnotesize,anchor=north east] (visibleText) at (9,-0.3) {not visible from above};
    \draw[->] ($(visibleText.north east)+(-0.45,0)$) -- ++(0,0.4);
    \end{tikzpicture}\vspace{-3mm}
  \caption{Upper vertices of Minkowski sums and regions of maxout arrangements. %
}
  \label{fig:my_label}
\end{figure}

\section{Number of linear regions for maxout networks}
\label{sec:main}
In this section we present our main results on the maximum number of linear regions of maxout networks.
First we provide %
general observations, then %
turn to our main results on shallow networks,
and derive implications for deep networks.
For shallow networks, we obtain sharp bounds and provide a construction that attains them.
The main analysis will be conducted in the upcoming Sections~\ref{sec:Weibel},~\ref{sec:Zaslavsky}, and~\ref{sec:Weibel-Zaslavsky}.

\subsection{General observations}
We begin with a simple general upper bound on the number of linear regions of the functions represented by maxout networks:

\begin{proposition}[Simple upper bound on the number of regions]
\label{prop:simple-upper-bound}
For any maxout network $\mathcal{N}$ with a total of $m$ maxout units of ranks $k_1,\ldots, k_m\in\mathbb{N}$,  $N(\mathcal{N})\leq \prod_{j=1}^m k_j$.
\end{proposition}
\begin{proof}
Let $n_0$ be the number of inputs. Let $L$ be the number of layers and denote their widths $n_1,\ldots, n_L$. Write $k_{l,j}$ for the rank of the $j$th unit $j=1,\ldots, n_l$ in the $l$th layer $l=1,\ldots,L$. Fix the parameters $A_{l,j,r}\in\mathbb{R}^{n_{l-1}}$, $b_{l,j,r}\in \mathbb{R}$ of all preactivation features $r=1,\ldots, k_{l,j}$, of all units $j=1,\ldots, n_l$, of all layers $l=1,\ldots,L$.
Then, for each input $x$, the represented function $f$ takes the form $(\bar A_{L} \cdots \bar A_{1}) x +  (\sum_{l=1}^L  \bar A_L \cdots \bar A_{l+1} \bar b_l)$, where each $\bar A_l \colon \mathbb{R}^{n_0}\to\mathbb{R}^{n_l\times n_{l-1}}$ and $\bar b_l\colon\mathbb{R}^{n_0}\to\mathbb{R}^{n_l}$ is a piecewise constant function of $x$ having $j$th row equal to one of the $k_{l,j}$ values $A_{l,j,1},\ldots,A_{l,j,k_{l,j}}\in\mathbb{R}^{n_{l-1}}$ and $b_{l,j,1},\ldots, b_{l,j,k_{l,j}}\in\mathbb{R}$, depending on which of the preactivation features assumes the maximum. %
The list of preactivation features that assume the maximum for each unit is called the activation pattern of the network at the particular input. %
The set of inputs with a particular activation pattern is determined by a list of linear inequalities and hence it is a convex polyhedron. %
In summary, the input space is split into at most
$\prod_{l=1}^L\prod_{i=1}^{n_l} k_{l,i}$ connected regions on each of which $f$ is linear.
\end{proof}

The following proposition states that generic perturbations of the parameters do not decrease the number of linear regions. Here, generic means up to a null set with respect to the Lebesgue measure in parameter space.
This is important, as later it will allow us to obtain sharp upper bounds by considering polytopes that are in general orientation.

\begin{proposition}[Generic perturbations of parameters do not decrease the number of regions]%
Consider a network $\mathcal{N}$ consisting of a finite number of maxout units.
Let $f_\theta\in \mathcal{N}$. Then there exists an $\epsilon=\epsilon(\theta)>0$ such that for %
generic $\theta'$ with $\|\theta'-\theta\|<\epsilon$, $N(f_{\theta'})\geq N(f_{\theta})$.
\end{proposition}
\begin{proof}
The intuition is that every linear region of $f_\theta$ contains a neighborhood of an input point $x_0$, and small perturbations of the parameter $\theta$ only cause small changes in the distance between $x_0$ and the nonlinear locus $V(f_\theta)$, so that no linear regions can `disappear' under small perturbations of the network parameters.
The formal argument is based on the correspondence between regions and upper vertices of the lifted Newton polytope, Proposition~\ref{prop:ftt}.
Consider the function $f_\theta$ represented by the network $\mathcal{N}$ with parameter $\theta$, and the corresponding convex function $\bar f_\theta=\max_j\{\langle a_j(\theta),x\rangle +b_j(\theta)\}$ described in~\eqref{eq:maxoutnetworkconvexfunction}.
The lifted Newton polytope of $\bar f_\theta$ is the convex hull of points $(a_j(\theta),b_j(\theta))\in\mathbb{R}^{n+1}$ that have a continuous, in fact polynomial, parametrization in $\theta$.
The statement now follows from the lower semi-continuity of the face numbers of polytopes discussed in  \cite[Section~5.3]{Gruenbaum2003}.
More precisely, denote by $\rho$ the Hausdorff metric which is defined as $\rho(A_1,A_2)=\inf\{\alpha>0\colon A_1\subseteq A_2 + B_\alpha, A_2\subseteq A_1 + B_\alpha \}$, where $B_\alpha$ is a radius-$\alpha$ ball around the origin.
If $P$ is a bounded polytope, then there exists an $\epsilon = \epsilon(P)>0$ such that for every $P'$ with $\rho(P', P)< \epsilon$ we have $f_k(P')\geq f_k(P)$ for any $0\leq k\leq n$, where $f_k(P)$ is the number of $k$-faces of $P$.
The same statement clearly applies to the upper faces of polytopes, and hence to the number of linear regions of $\bar f_\theta$.
Since the linear regions of $f_\theta$ and $\bar f_\theta$ are equal for generic choices of parameters, the claim follows.  %
\end{proof}

\subsection{Shallow networks}
For shallow maxout networks, \cite{NIPS2014_5422} obtained the following bounds. The upper bound is based on embedding a maxout arrangement in a hyperplane arrangement and using well known upper bounds for that case. We add a minor improvement which was pointed out in \cite{DBLP:conf/icml/SerraTR18} (substituting $k^2$ with $k(k-1)/2$).

\begin{proposition}[{\cite[Proposition 7]{NIPS2014_5422}}]
\label{proposition:previous-bounds}
For a network $\mathcal{N}$ with $n$ inputs and a single layer of $m$ rank-$k$ maxout units, $k^{\min\{n,m\}}\leq N(\mathcal{N})\leq \sum_{j=0}^{n}{m k (k-1)/2 \choose j}$.
\end{proposition}

Notice the significant gap between the lower and upper bounds in Proposition~\ref{proposition:previous-bounds}, of orders $\Omega(k^{n})$ and $O((m k^2)^n)$ in $m$ and $k$.
The construction for the lower bound can 
be generalized and improved, as we show in the next proposition. In Theorem~\ref{thm:main-result} we will show that this lower bound is optimal.

\begin{proposition}[Lower bound for shallow maxout networks]
\label{prop:singleLayerLowerBound}
For a network $\mathcal{N}$ with $n$ inputs and a single layer of $m$ maxout units of ranks $k_1,\ldots,k_m$,
$N(\mathcal{N})\geq\sum_{j=0}^{n} \sum_{S\in {[m]\choose j}} \prod_{i\in S}(k_i-1)$.
The bound is realized if each of the maxout units has a nonlinear locus consisting of $k_i-1$ distinct parallel hyperplanes and the normals of different units are in general position.
\end{proposition}

\begin{proof}
We count the number of regions for the special case where each maxout unit has a nonlinear locus consisting of $k_i-1$ parallel hyperplanes, $i=1,\ldots,m$.
This provides a lower bound on the maximum possible number of regions.
Zaslavsky's theorem \cite{zaslavsky1975facing} states that the number of regions of an arrangement $\mathcal{A}$ of affine hyperplanes in an $n$-dimensional real vector space is $r(\mathcal{A}) = (-1)^n\chi_\mathcal{A}(-1)$, where $\chi_\mathcal{A}$ is the characteristic polynomial of $\mathcal{A}$.
For \emph{generic translations of hyperplanes of a linear arrangement}, Stanley \cite[pg.\ 22]{Stanley04anintroduction} shows that $\chi_\mathcal{A}(t) = \sum_\mathcal{B}(-1)^{|\mathcal{B}|} t^{n-|\mathcal{B}|}$, where $\mathcal{B}$ ranges over all subsets of hyperplanes in $\mathcal{A}$ with linearly independent normals.
Applying these two results to an arrangement $\mathcal{A}$ in $\mathbb{R}^{n}$ consisting of $m$ sets of $k_i-1$ parallel hyperplanes, $i=1,\ldots, m$, with hyperplanes in different sets being in general position,
we obtain
$\chi_\mathcal{A}(t) = \sum_{j=0}^{n} (\sum_{\subsmash{S\in {[m]\choose j}}} \prod_{i\in S}(k_i-1))(-1)^j t^{n-j}$ and
$r(\mathcal{A})= (-1)^n\chi_\mathcal{A}(-1)
= \sum_{j=0}^{n} (\sum_{S\in {[m]\choose j}} \prod_{i\in S}(k_i-1))$.
\end{proof}

A similar argument can be used to obtain the following lower bound for the maximum number of regions for functions represented by maxout networks without biases. In Theorem~\ref{thm:main-result} we will show that this lower bound is also optimal.
\begin{proposition}[Lower bound for shallow maxout networks without biases]
\label{prop:singleLayerLowerBoundNoBias}
For a network $\mathcal{N}$ with $n$ inputs and a single layer of maxout units of ranks $k_1,\ldots, k_m$ and no biases, $N(\mathcal{N})\geq {m-1\choose n-1} + \sum_{j=0}^{n-1} \sum_{S\in{[m]\choose j}}\prod_{i\in S}(k_i-1)$.
\end{proposition}

\begin{proof}
Consider $\mathbb{R}^n$ and the hyperplane $H = \{x'\in\mathbb{R}^n\colon x'_n=1\}$.
Any linear function $x'\mapsto \langle w', x' \rangle$ on $\mathbb{R}^n$ takes over $H$ the form $\langle  w , x \rangle + b $, where $w' = (w,b)$ and $x'=(x,1)$.
Thus, setting the weights of our layer without biases as $w_{i}'=(w_{i},b_{i})\in\mathbb{R}^n$,
with $w_i\in\mathbb{R}^{n-1}$ the weights
and $b_i\in\mathbb{R}$ the biases of Proposition~\ref{prop:singleLayerLowerBound} for $n-1$ inputs, where $i$ runs over all preactivation features of all units, we obtain a function whose restriction to $H$ has $\sum_{j=0}^{n-1}\sum_{S\in{[m]\choose j}}\prod_{i\in S}(k_i-1)$ linear regions. %
Now we argue that this function can be constructed so that it has ${m-1\choose n-1}$ additional linear regions that are not intersected by $H$.
By our construction, over $H$ the nonlinear locus of the $j$th unit $\max\{\langle w_{j,1},x\rangle +b_{j,1},\ldots, \langle w_{j,k_j},x\rangle +b_{j,k_j}\}$ consists of $k_j-1$ parallel hyperplanes, meaning that all $w_{j,1},\ldots, w_{j,k_j}$ are equal to some fixed $w_j$ up to scaling.
Consider the hyperplane $G=\{x'\in\mathbb{R}^n\colon x'_n=-1\}$. Over $G$, any of the maxout units takes the form $\max\{\langle \alpha_1 w_{j},x\rangle -b_{j,1} ,\ldots, \langle \alpha_{k_j} w_{j},x\rangle -b_{j,k_j}\}$ and its nonlinear locus includes (actually it consists precisely of) one hyperplane $\{x\in \mathbb{R}^{n-1}\colon (\alpha_r-\alpha_s)\langle w_j ,x\rangle - (b_{j,r}-b_{j,s})\}$.
Hence, over $G$ our layer 
has linear regions determined by $m$ affine hyperplanes. Now, an arrangement of $m$ hyperplanes in general position in ${n-1}$ dimensions has ${m-1\choose n-1}$ relatively bounded regions.
None of these intersect $H$.
\end{proof}

Our main result determines the maximal number of linear regions for shallow maxout networks with and without bias, for any input dimension, any number of maxout units, and any ranks.
It shows that the lower bounds in Propositions~\ref{prop:singleLayerLowerBound} and~\ref{prop:singleLayerLowerBoundNoBias} are sharp.

\begin{theorem}[Optimal bound for shallow maxout networks]
\label{thm:main-result}
For a shallow network $\mathcal{N}$ with $n$ inputs and a layer of $m$ maxout units of ranks $k_1,\ldots, k_m$, we have
\begin{align*}
N(\mathcal{N}) &= \sum_{j=0}^{n}  \sum_{S\in {[m]\choose j}} \prod_{i\in S}(k_i-1),\\
\intertext{where ${[m]\choose j}$ is the set of subsets of $[m]=\{1,\ldots, m\}$ with cardinality $j$.
Here ${[m]\choose 0} = \{\emptyset\}$, empty sums are $0$, and empty products are $1$.
In particular, if $k_1=\cdots=k_m =k$, we have $N(\mathcal{N}) = \sum_{j=0}^{n}{m\choose j}(k-1)^j$.
Moreover, if $\mathcal N$ is without biases, then}
N(\mathcal{N}) &= {m'-1\choose n-1} + \sum_{j=0}^{n-1}  \sum_{S\in {[m]\choose j}} \prod_{i\in S}(k_i-1),
\end{align*}
where $m'$ is the number of maxout units of rank larger than $1$.
In particular, if $k_1=\cdots=k_m =k>1$, we have $N(\mathcal{N}) = {m-1\choose n-1}+\sum_{j=0}^{n-1}{m\choose j}(k-1)^j$.
\end{theorem}

The proof relies on several results that will be developed in Sections~\ref{sec:Weibel}, \ref{sec:Zaslavsky}, and \ref{sec:Weibel-Zaslavsky}.
The most difficult part %
is the upper bound for the case with biases and many units ($m\geq n+1$)  possibly having small ranks (allowed to be smaller than $n+2$), which also happens to be the case of highest practical interest.
The main ideas are as follows.

It is not difficult to adapt a result by Weibel to count upper faces of Minkowski sums (Theorem~\ref{thm:f_vectors_upper_part_Minkowski_sum}).
This provides us with a formula for the number of linear regions (and other lower-dimensional features) for shallow maxout networks. However, the formula consists of an alternating sum over sub-arrangements whose maximum value is difficult to determine when some of the polytopes are not full dimensional.

Studying whole polytopes instead of their upper faces allows us to leverage Adiprasito-Sanyal's Upper Bound Theorem for Minkowski Sums and its implications for small Minkwoski subsums (Proposition~\ref{thm:Mneighborly_upper_bound}).
To obtain an explicit formula for the maximum number of overall vertices, we generalize Weibel's formula for vertices to encompass possibly lower-dimensional polytopes (Theorem~\ref{thm:central-arbdims}) and upper vertices (Theorem~\ref{thm:facessimple}). To this end we formulate a maxout version of Zaslavsky's theorem (Theorem~\ref{thm:posetcounting}).
In order to differentiate between all vertices and upper vertices, we study bounded regions of maxout arrangements (Theorem~\ref{thm:lowerbound_strict_lower}).
In summary, the key results towards proving Theorem~\ref{thm:main-result} are the following:
\setlist[description]{font=\normalfont}
\begin{description}[leftmargin=3mm]
\item[Propositions~\ref{prop:singleLayerLowerBound} and \ref{prop:singleLayerLowerBoundNoBias}:] constructive lower bounds for the maximum number of linear regions of shallow maxout networks with and without biases, improving a previous construction by Mont\'ufar et al.~\cite{NIPS2014_5422}.
  \item[Proposition~\ref{thm:Mneighborly_upper_bound}:] a consequence of the Upper Bound Theorem for Minkowski sums by Adiprasito-Sanyal~\cite{KarimRaman} to the number of vertices of small Minkowski subsums of polytopes attaining said upper bound.
  \item[Theorem~\ref{thm:posetcounting}:] a maxout version of Zaslavsky's theorem for hyperplane arrangements~\cite{zaslavsky1975facing}, which expresses the number of regions of a maxout arrangement in terms of the Euler characteristic and the M\"obius function on the intersection poset.
  \item[Theorems~\ref{thm:facessimple} and~\ref{thm:central-arbdims}:] a generalization of Weibel's counting formula for Minkowski sums~\cite{Weibel12}, which expresses the number of regions of simple non-central and central maxout arrangements in terms of the regions of small subarrangements (Minkowski subsums).
  \item[Theorem~\ref{thm:lowerbound_strict_lower}:] a lower bound on the number of bounded regions of a maxout arrangement (strict lower vertices of a Minkowski sum). This allows us to upper bound the number of upper vertices  of a Minkowski sum, given an upper bound on the total number of vertices.
\end{description}

\begin{proof}[Proof of Theorem~\ref{thm:main-result}]
For networks with biases, if $m\leq n$, we can apply the trivial upper bound given in Proposition~\ref{prop:simple-upper-bound}.
For $m \geq n+1$, the upper bound follows from the upper bound theorem for Minkowski sums of polytopes  Theorem~\ref{thm:Mneighborly_upper_bound} together with the lower bound on the number of strict upper or lower faces given in  Theorem~\ref{thm:lowerbound_strict_lower} inserted into the counting formula given in Theorem~\ref{thm:facessimple} and reformulated via Lemma~\ref{lemma:reformulation}.
The construction attaining the maximum is given in Proposition~\ref{prop:singleLayerLowerBound}.
For networks without biases, if $m\leq n$,
we can apply the trivial upper bound given in Proposition~\ref{prop:simple-upper-bound}.
For $m\geq n+1$, the upper bound follows from
Theorem~\ref{thm:Mneighborly_upper_bound} inserted into the counting formula given in Theorem~\ref{thm:central-arbdims} and reformulated via Lemma~\ref{lemma:reformulation}.
The construction attaining the maximum is given in Proposition~\ref{prop:singleLayerLowerBoundNoBias}.
\end{proof}

We illustrate Theorem~\ref{thm:main-result} on a few examples.

\begin{example}\
\begin{enumerate}[leftmargin=*]
\item
In the case of a single input, $n=1$, networks with biases represent functions on the real line which have at most $\sum_{i=1}^m (k_i-1)$ break points %
and a maximum of $1 + \sum_{i=1}^m (k_i-1)$ linear regions.
Networks without biases and at least one unit of rank $\geq2$ represent functions which have at most $1$ break point and a maximum of %
$2$ linear regions.

\item
In the case of few units, $m\leq n$, networks with and without biases both have the optimal bound $\prod_{i=1}^m k_i$.
To see this, use the multi-binomial theorem to evaluate the formula in our theorem.
This matches the simple upper bound given in Proposition~\ref{prop:simple-upper-bound}.

\item
In the case of many units, $m\geq n$, the maximum number of regions is no longer exponential in $m$, but only polynomial.
For $k_1=\cdots=k_m=k$, the order in $m$ and $k$ is $\Theta((mk)^n)$ and $\Theta((mk)^{n-1})$ in the cases with and without biases.
This should be compared with the previous bounds $\Omega(k^{n})$ and $O((m k^2)^n)$ from Proposition~\ref{proposition:previous-bounds} for the case with biases.

\item
In the case $k_1=\cdots=k_m=2$,
we recover the well-known formulas for the maximum number of regions of hyperplane arrangements, $\sum_{j=0}^n{m\choose j}$,
and central hyperplane arrangements,
${m-1\choose n-1}+\sum_{j=0}^{n-1}{m\choose j}
= 2\sum_{j=0}^{n-1}{m-1\choose j}$.
These are also the optimal bounds for shallow ReLU networks with and without biases.
\end{enumerate}
\end{example}

We finish this part with two %
corollaries.
The first gives a lower bound on the number of bounded regions.

\begin{corollary}[Lower bound on the number of bounded regions]
\label{cor:lowervert}
For a network with $n$ inputs, a layer of $m$ maxout units of ranks at least $2$, and generic parameters,
the number of bounded regions is at least ${m-1\choose n}$.
For networks without biases all regions are unbounded. %
\end{corollary}
\begin{proof}
This follows from Theorem~\ref{thm:lowerbound_strict_lower}.
\end{proof}

By Corollary~\ref{cor:lowervert}, generic maxout arrangements with rank at least $2$ have at least as many bounded regions as generic hyperplane arrangements, which have ${m-1\choose n}$ bounded regions.
This observation is non-trivial, %
since the polyhedral pieces of %
tropical hypersurfaces do not necessarily all intersect each other.
It is easy to draw examples in $\mathbb{R}^2$ showing that, in contrast to hyperplane arrangements, maxout arrangements do not have a single generic number of bounded regions.
Lower bounds for generic parameters are rare in the literature.
A result of this kind is \cite[Corollary 8.2]{Adiprasito2017} (supplement to \cite{KarimRaman}), which shows that a Minkowski sum of $m$ polytopes in general position has at least as many vertices as a sum of $m$ line segments.

The following is a simple corollary for the number of regions over an affine subspace of the input space, which we will use in the next part on deep networks.

\begin{corollary}[Number of regions over an affine subspace]
\label{cor:convexRegionsOfRestriction}
Consider a network $\mathcal{N}$ with $n$ inputs and a layer of $m$ maxout units. Let $A$ be an affine $n_0$-space, $n_0\leq n$.
Then $N (\mathcal{N}|_A) = \sum_{j=0}^{n_0}\sum_{S\in{[m]\choose j}}\prod_{i\in S}(k_i-1)$.
Similarly, for a network $\mathcal{N}$ without biases and $A$ a linear $n_0$-space, $n_0\leq n$, %
$N(\mathcal{N}|_A) = {m-1\choose n_0-1} +  \sum_{j=0}^{n_0-1}\sum_{S\in{[m]\choose j}}\prod_{i\in S}(k_i-1)$.
\end{corollary}
\begin{proof}
The functions represented by $\mathcal{N}$ on $A$ can be written as a layer with $n_0$ inputs.
\end{proof}

\subsection{Deep networks}

In this subsection we derive consequences of our analysis of shallow networks for deep networks.
For deep maxout networks, \cite{NIPS2014_5422} obtained the following lower bound. The main point in that work was to show that the maximum number of linear regions is exponential in the depth of the network.
Upper bounds for deep networks can be obtained by multiplying upper bounds for individual layers, whereby the effective input dimension of each layer is bounded by the dimension of the image of the previous layers %
\cite{montufar2017notes}.
The following upper bound of this form was given in~\cite{DBLP:conf/icml/SerraTR18}, whereby we correct a minor typo (the sum runs up to $\min\{n_0,\ldots,n_{l-1}\}$ rather than $\min\{n_0,\ldots,n_{l}\}$).

\begin{proposition}[{\cite[Theorem~9]{NIPS2014_5422} and \cite[Theorem 10]{DBLP:conf/icml/SerraTR18}}]
For a rank-$k$ maxout network $\mathcal{N}$ with $n_0$ inputs and $L$ layers of width $n_0$, $N(\mathcal{N})\geq k^{L-1}k^{n_0}$.
For %
$L$ layers of widths $n_1,\ldots, n_L$, $N(\mathcal{N})\leq \prod_{l=1}^L(\sum_{j=0}^{e_l} {n_l k(k-1)/2\choose j })$, where $e_l=\min\{n_0,\ldots, n_{l-1}\}$.
\end{proposition}

We observe that there is a significant gap between the lower and upper bounds, of orders $\Omega(k^{L-1 + n_0 })$ and $O(\prod_{l=1}^L (n_l k^2)^{n_0})$ in %
$n_1,\ldots, n_L$ and %
$k$.
We can refine the approach from \cite{NIPS2014_5422} to obtain the following lower bound of order $\Omega(\prod_{l=1}^L(n_lk)^{n_0})$   %
in $n_1,\ldots, n_L$ and $k$.
This not only grows exponentially with the depth $L$, but also grows with the layer widths.

\begin{proposition}[Lower bound for deep maxout networks]
\label{prop:deep-lower}
Consider a network $\mathcal{N}$ with $n_0$ inputs and $L$ layers of $n_1,\ldots, n_L$ rank-$k$ maxout units.
Let $n\leq n_0, \frac12 n_1,\ldots, \frac12n_{L-1}$.
Assume $\frac{n_l}{n}$ is even (else take the largest even lower bound and discard the rest).
Then $N(\mathcal{N})\geq (\prod_{l=1}^{L-1} (\frac{n_l}{n} (k-1)+1)^n) (\sum_{j=0}^n{n_L\choose j}(k-1)^j)$.
For the same network but without biases, assuming $\frac{n_l-1}{n-1}$ is even, $N(\mathcal{N})\geq (\prod_{l=1}^{L-1} (\frac{n_l-1}{(n-1)} (k-1)+1)^{n-1}) (%
\sum_{j=0}^{n-1}{n_L\choose j}(k-1)^j)$.
\end{proposition}

\begin{proof}
We follow the general arguments from \cite{NIPS2014_5422} but modify the construction of the weights %
to be similar to the one used in the same paper for ReLU layers.
The idea is to construct a many-to-one function, which allows us to multiply the regions across layers.

Consider first the case with biases. We consider the restriction of the network to inputs from a subspace of dimension $n$. Further, we insert a linear layer of output dimension $n$ after each layer of maxout units. In this way, the input dimension for each layer is $n$. This does not increase the representational power of the network, since a linear layer can be subsumed into the input weights and biases of the next layer, as $A_{i+1}(B_i f_i(x) + c_i)+b_{i+1} = (A_{i+1} B_i) f_i(x) + (A_{i+1}c_i +b_{i+1})$.
For each layer $l=1,\ldots,L-1$, we organize the $n_l$ units into $n$ groups of even size $\frac{n_l}{n}$.
By choosing the parameters of the $i$th group appropriately, we can achieve that their sum with alternating signs represents a zig-zag function over $\mathbb{R}^n$ with $\frac{n_l}{n}(k-1)$ breakpoints along the $i$th coordinate. To see this, note that along any given direction of its input, a maxout unit can represent any piecewise linear convex function with $k-1$ break points.
This way we can achieve that the $i$th layer maps $[0,1]^n$ to $[0,1]^n$ in a $(\frac{n_l}{n}(k-1)+1)^n$ to one manner. The function computed up to layer $L-1$ multiplies these multiplicities.
By Proposition~\ref{prop:singleLayerLowerBound}, the last layer can create $\sum_{j=0}^n{n_L\choose j}(k-1)^j$ regions over the an $n$-dimensional subspace of its input space, which, by appropriate scaling will intersect $[0,1]^n$. Each of these regions has multiplicity $\prod_{l=1}^{L-1}(\frac{n_l}{n}(k-1)+1)^n$ over the input space of the network, thus giving the indicated lower bound.

Consider now the case without biases. %
For each layer $l=1,\ldots,L$, we choose the weights of all preactivation features of the $n_{l}$th unit as the coordinate vector $e_{n_{l-1}}\in\mathbb{R}^{n_{l-1}}$, so that
$x^l_{n_l}=\max\{x^{l-1}_{n_{l-1}},\ldots, x^{l-1}_{n_{l-1}}\}=x^{l-1}_{n_{l-1}}$.
We consider the restriction of the network to the subset of inputs given by the hyperplane $H = \{x^0\in\mathbb{R}^{n_{0}} \colon x^{0}_{n_{0}}=1\}$. The number of linear regions of a function over this subset is a lower bound on its number of regions over the entire input space $\mathbb{R}^{n_0}$.
Notice that, given our choice of weights, over $H$ the last unit of each layer takes the fixed value $x^0_{n_0}=x^1_{n_1}=\cdots=x^L_{n_L}=1$.
As in Proposition~\ref{prop:singleLayerLowerBoundNoBias}, we choose the weights of the units $i=1,\ldots, n_l-1$ in layer $l$ as $w'_i = (w_i,b_i)\in\mathbb{R}^{n_{l-1}\times k}$, with $w_i\in\mathbb{R}^{(n_{l-1}-1)\times k}$ the weights and $b_i\in\mathbb{R}^{1\times k}$ the biases that are used above for a layer with biases and $n_{l-1}-1$ inputs.
Hence, over $H$ we obtain the same many-to-one maps as above, but now with the widths substituted to $n_0-1,\ldots, n_{L-1}-1$.
Finally, note that the last layer can in fact be chosen as in Proposition~\ref{prop:singleLayerLowerBoundNoBias} with $n$ inputs and $n_L$ outputs.
We take the bound $\sum_{j=0}^{n-1}{n_L\choose j}(k-1)^j$ for the number of regions intersecting $\{x^{L-1}\in\mathbb{R}^{n_{L-1}}\colon x^{L-1}_{n_{L-1}}=1\}$ and ignore other regions.
\end{proof}
Proposition~\ref{prop:deep-lower} is given for networks where all units have the same rank $k$, but it is straightforward to formulate corresponding results for networks with units of different ranks.
Also, it is not difficult to obtain minor improvements if instead of discarding units one keeps them with small weights, without altering the general construction. However, as we will see below, the asymptotic is already tight.

We now derive upper bounds the number of linear regions $N(\Phi)$ of a
function $\Phi:\RR^{n_0}\to\dots\to\RR^{n_L}$ represented by a deep neural network.
Notice that each linear region of the function computed up to the $(l-1)$th layer is split by the $l$th layer into at most %
the number of linear regions of a shallow network with $n_{l-1}$ inputs and $n_l$ outputs.
As pointed out in \cite{montufar2017notes}, the linear output pieces of the $(l-1)$th layer have dimension bounded above by $\min\{n_0,\ldots, n_{l-1}\}$, which allows us to slightly improve the bound based on Corollary~\ref{cor:convexRegionsOfRestriction}.
Similar discussions have also appeared in \cite{DBLP:conf/icml/SerraTR18} and \cite[Theorem 6.3]{pmlr-v80-zhang18i}.
We obtain the following upper bound for deep networks. When all units have rank $k$, the bound is of order $O(\prod_{l=1}^L(n_l k)^{n_0})$ in the layer widths $n_1,\ldots, n_L$ and rank~$k$, which in view of Proposition~\ref{prop:deep-lower} is tight.

\begin{theorem}[Upper bound for deep maxout networks]
\label{thm:deep-result}
Let $\mathcal{N}$ be an network with $n_0$ inputs and $L$ layers of $n_l$ maxout units of ranks $k_{l,1},\ldots, k_{l,n_l}$, $l=1,\ldots, L$.
Let $e_l=\min\{n_0,\dots, n_{l-1}\}$, $l=1,\ldots, L$.
Then
\begin{align*}
N(\mathcal{N})&\le \prod_{l=1}^L \sum_{j=0}^{e_l}\sum_{S\in{[n_l]\choose j}}\prod_{i\in S} (k_{l,i}-1).\\
\intertext{For the same network but without biases,}
N(\mathcal{N})&\le \prod_{l=1}^L
\Big( {n_l-1\choose e_{l}-1} + \sum_{j=0}^{e_{l}-1}\sum_{S\in{[n_l]\choose j}}\prod_{i\in S} (k_{l,i}-1)\Big).
\end{align*}
Moreover, these bounds are asymptotically sharp.
\end{theorem}

\begin{proof}
We write $\Phi^{(i,j)}:\RR^{n_{i-1}}\to \dots\to \RR^{n_j}$ for the function represented by the network consisting of layers $i,\ldots, j$, and write $\Phi^{(j)}$ for $\Phi^{(1,j)}$.
For any $l\in[L]$ we have that $\Phi = \Phi^{(l,L)}\circ \Phi^{(l-1)}$.
We consider $N_c(\Phi)$, the smallest number of convex regions that form a refinement of the linear regions of the function $\Phi$.
For a positive integer $e\le n_0$, we will also need to consider
\[
N_c(\Phi\mid e):= \max\{ N_c(\Phi|_\Omega): \Omega\subseteq \RR^{n_0}\text{ is an }e\text{-dimensional affine space}\}.
\]
Now, if $F$ is a layer with output dimension $d$,
and $H= G\circ F$, where $G$ is a layer with a compatible number of inputs, then
$ N(H) \le N_c(H) \le  N_c(G\mid d)\cdot N_c(F)$. See \cite[Theorem D.3]{pmlr-v80-zhang18i} for a discussion.
Hence
$N(\Phi)\le  N_c(\Phi^{(l,L)}\mid e_l) \cdot  N_c(\Phi^{(l-1)}) =  N_c(\Phi^{(l,L)})\cdot  N_c(\Phi^{(l-1)})$.
The bounds then follow by induction %
and Corollary~\ref{cor:convexRegionsOfRestriction}, which bounds the number of regions of a layer with inputs from an affine space of given dimension.
Finally, the asymptotic tightness of the bounds follows in view of Proposition~\ref{prop:deep-lower}.
\end{proof}

The bound is based on the observation that each of the linear regions of a network with $l-1$ layers is mapped to a polyhedron of dimension at most $e_l$ in the input space of the $l$th layer.
The $l$th layer will split each of these polyhedra into at most as many regions as it can create over an affine space of dimension $e_l$.
In principle, one can pursue a more refined analysis by recursively investigating the arrangement that is induced by the $l$th layer on the graph of the $(l-1)$-layer network, similar to the analysis that we conduct in Sections~\ref{sec:Zaslavsky} and \ref{sec:Weibel-Zaslavsky} for maxout arrangements.
\section{Face counting formulas \`a la Weibel}
\label{sec:Weibel}

Weibel \cite{Weibel12} obtained a counting formula for the number of faces of large Minkowski sums of full-dimensional polytopes $P_1,\ldots, P_m\subseteq\mathbb{R}^{n+1}$, $m\geq n+1$, $\dim(P_i)=n+1$,
in terms of the numbers of faces of partial sums of up to $n$ of the polytopes.
In the following we give a similar formula for the number of \emph{upper} faces, which also holds when the summands have arbitrary dimensions.
The idea of \cite{Weibel12} is to enumerate the $s$-faces of %
$P=P_1+\cdots+P_m$ by inclusion-exclusion of the $s$-faces of the partial sums $P_S=\sum_{i\in S}P_i$, $S\subseteq\{1,\ldots, m\}$ with $1\leq |S|\leq n$.
In order to do this, polytopes are associated with spherical complexes, and cells of the complex are assigned a witness \emph{westernmost corner}.
We use similar definitions with slight modifications.

\begin{definition}[Spherical complex and upper complex of a polytope]
Let $P\subseteq\RR^{n+1}$ be a polytope. To each face $F$ of $P$ we associate the cell of directions it maximizes: $C(F,P)=\{l\in \mathbb{S}^{n}\colon \langle l, x -y\rangle > 0\; \forall x\in F, y\in P\setminus F\}$. The set of all such cells is a spherical complex $\mathcal{G}(P)$ dual to $P$.
The upper complex $\mathcal{G}^+(P)$ consists of the intersections of cells in $\mathcal{G}(P)$ with $(\mathbb{R}^{n}\times \mathbb{R}_{>0})$. The upper part $P^+$ of a polytope $P$ in $\mathbb{R}^{n+1}$, is the collection of faces $F$ of $P$ %
whose cells $C(F,P)$ intersect $(\mathbb{R}^{n}\times \mathbb{R}_{>0})$.
Let $\mathbb{S}^n_{\geq0} = \mathbb{S}^{n}\cap(\mathbb{R}^{n}\times\mathbb{R}_{\geq0})$.
\end{definition}

\begin{definition}[General orientation]
  A polytope $P\subseteq\RR^{n+1}$ is said to be in general orientation (relative to our coordinates) if none of the great circles defined by one-dimensional cells in $\mathcal{G}(P)$ contains a standard unit vector. A family of polytopes $P_1,\dots,P_m\subseteq\RR^{n+1}$ is in general orientation, if each $P_i$ is in general orientation and for all $S\subseteq [m]$ and any $C_i\in \mathcal{G}(P_i)$, $i\in S$,
  the intersection
  $\bigcap_{i\in S} C_i$ is either empty or has codimension at least $\min\{\sum_i \operatorname{codim}(C_i),n\}$. %
\end{definition}

For $i=n+1,n-1,n-3,\ldots$, let $U^{i}$ be the $i$-dimensional %
subspace $\operatorname{span}\{e_{n+1-i+1},\ldots,e_{n+1}\}$ if $i\geq 1$ and just the zero space $\{0\}$ if $i\leq 0$. Given a fixed $U^i$, we %
write $\hat e_{k} = e_{n+1-i+1+k}$ for $1\leq k\leq i$ so that $\hat e_{1},\dots,\hat e_{i}$ is a basis of it.

\begin{definition}[Direction west]\label{def:directionWest}
Let $i>1$.
At every point in $\mathbb{S}^n\cap U^i \cong \mathbb{S}^{i-1}$ and not in  $\mathbb{S}^{n}\cap U^{i-2} \cong \mathbb{S}^{i-3}$
define the direction west as the direction of increasing $\theta_1$ in the coordinate system given by
$\mathbb{S}^1 = \{\sin(\theta_1)\hat e_{1} + \cos(\theta_1) \hat e_{2} \colon \theta_1\in[0,2\pi)\}$ and
$\mathbb{S}^k = \{\sin(\theta_k) \mathbb{S}^{k-1} + \cos(\theta_k)\hat e_{k+1}\colon \theta_{k+1}\in[0,\pi]\}$ for $k=2,3,\ldots,i-1$.
 Here, $\hat e_{k}=e_{n+1-i+1+k}$.
\end{definition}

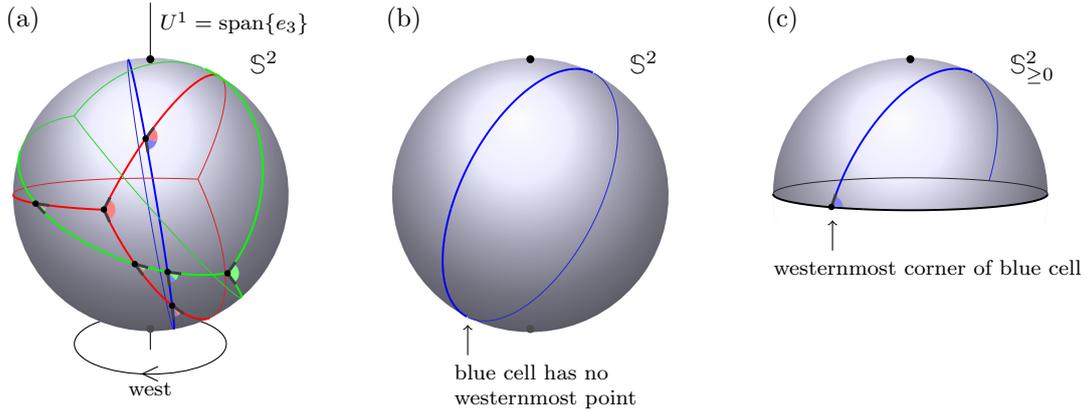
\begin{figure}
    \centering
    \begin{tikzpicture}
      \draw (0,-3.925) -- ++(0,-0.25);
      \draw (0,-4.1) ellipse (10mm and 4mm);
      \node[anchor=north] (westernmostCorners) at (0,0) {\includegraphics[width=4cm]{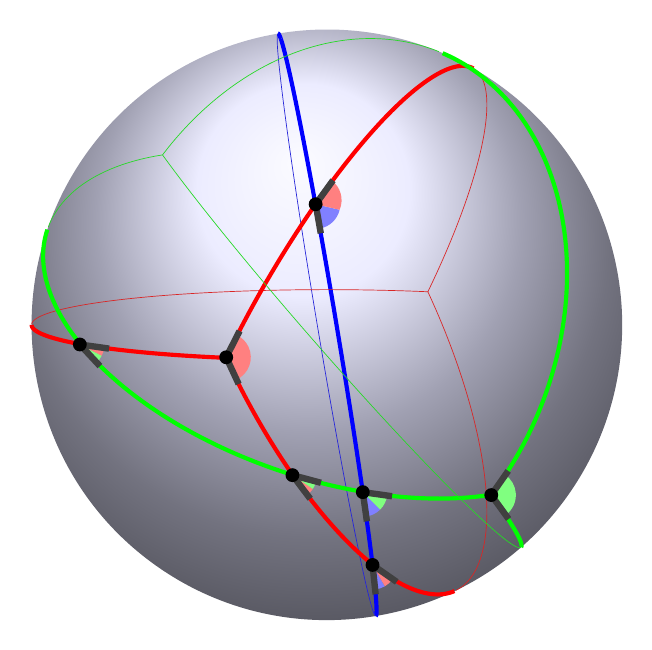}};
      \node at (0,-4.5) {$<$};
      \node[anchor=south,font=\footnotesize] at (0,-4.9) {west};
      \node[anchor=north west,xshift=1mm,yshift=5mm] at (westernmostCorners.north west) {(a)};
      \fill (0,-0.325) circle (1.5pt); %
      \draw (0,-0.325) -- ++(0,0.75) node[anchor=north west,font=\footnotesize] {$U^1=\text{span}\{e_3\}$};
      \node[anchor=north east,xshift=-1mm,yshift=-1mm] at (westernmostCorners.north east) {$\mathbb S^2_{\phantom{\geq 0}}$};
      \fill[black!70] (0,-3.9) circle (1.5pt);

      \fill (5,-3.9) circle (1.5pt);
      \node[anchor=north] (fullSphere) at (5,0)
      {\includegraphics[width=4cm]{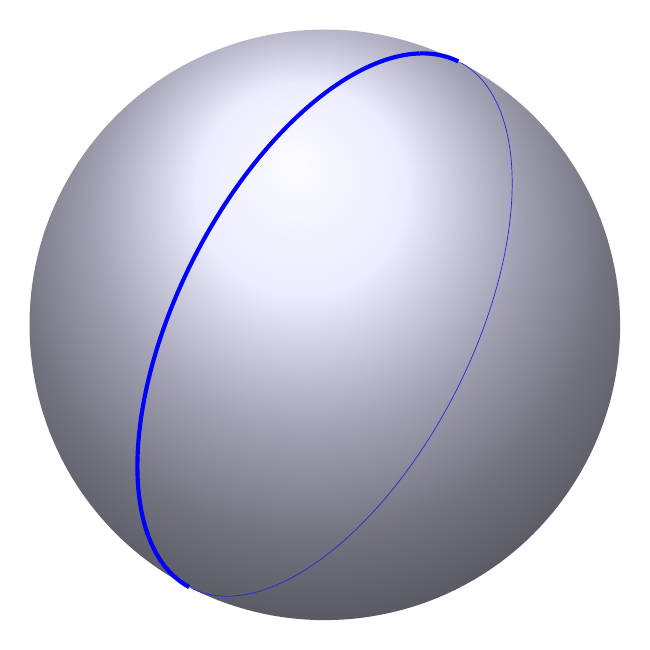}};
      \node[anchor=north east,xshift=-1mm,yshift=-1mm] at (fullSphere.north east) {$\mathbb S^2_{\phantom{\geq 0}}$};
      \node[anchor=north west,xshift=1mm,yshift=5mm] at (fullSphere.north west) {(b)};
      \fill (5,-0.325) circle (1.5pt);
      \fill[black!70] (5,-3.9) circle (1.5pt);

      \node[anchor=north] (halfSphere) at (10,0)
      {\includegraphics[clip=true, trim=0cm 2.6cm 0cm 0cm,width=4cm]{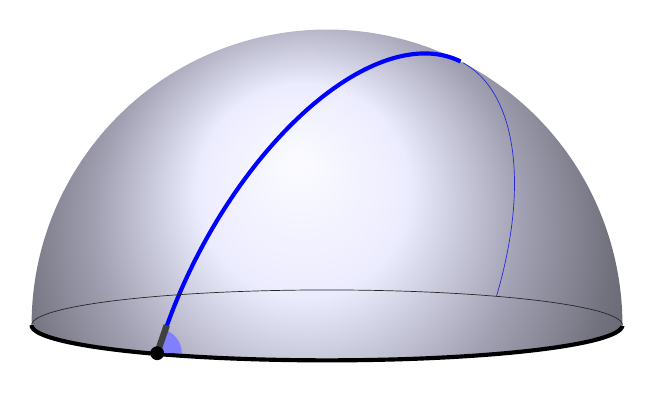}}; %
      \node[anchor=north east,xshift=-1mm,yshift=-1mm] at (halfSphere.north east) {$\mathbb S^2_{\geq 0}$};
      \node[anchor=north west,xshift=1mm,yshift=5mm] at (halfSphere.north west) {(c)};
      \fill (10,-0.325) circle (1.5pt);

      \node[anchor=north west,xshift=10mm,yshift=0mm,font=\footnotesize,text width=50mm] (fullSphereText) at (fullSphere.south west) {blue cell has no\\ westernmost point};
      \draw[->] ($(fullSphereText.north west)+(0.3,0)$) -- ++(0,0.4);
      \node[anchor=north west,xshift=2mm,yshift=-2mm,font=\footnotesize] (halfSphereText) at (halfSphere.south west) {westernmost corner of blue cell};
      \draw[->] ($(halfSphereText.north west)+(0.9,0)$) -- ++(0,0.4);
    \end{tikzpicture}\vspace{-3mm}
    \caption{(a) Shown is $\mathcal G(P_1+P_2+P_3)$, where $P_1$, $P_2$, $P_3$ have $2$, $3$, $3$ vertices respectively, and the westernmost corners of its $0$-, $1$-, and $2$-cells (westernmost corners of $2$-cells are colored by their support); (b) Cells of $\mathcal G(P)$, $P$ lower-dimensional, may not 
    have westernmost points, but
    (c) cells of $\mathcal G^+(P)$ do.}
    \label{fig:westernmostCorners}
\end{figure}

\begin{definition}[Westernmost point and westernmost corner]
We define a westernmost point of a cell $C\subseteq\mathbb{S}^n$ as follows:
Let $U^i$ be the smallest subspace in the sequence $U^{n+1},U^{n-1},\ldots$ which has a nonempty intersection with $C$. If $i>1$, a westernmost point of $C$ is a local optimizer of the direction west 
in the closure of $C\cap U^i$.
If $i=1$, then $C\cap U^i\subseteq\{\pm e_{n+1}\}$.
For $C\cap U^i=\{\pm e_{n+1}\}$ the westernmost point of $C$ is $e_{n+1}$, otherwise it is the single point in the intersection.
Finally, we define a westernmost corner of a cell as the intersection of its closure with a small ball around a westernmost point.
\end{definition}

The definition is illustrated in Figure~\ref{fig:westernmostCorners} (a). Example~\ref{example:westmostPoints} illustrates the existence and non-existence of westernmost corners. Lemma~\ref{lemma:existence} will state sufficient conditions for existence and uniqueness.

\begin{example}\label{example:westmostPoints}\
\begin{enumerate}[leftmargin=*]
\item Consider the upper sphere $C=\mathbb{S}^n_{\geq0}$.
If $n$ is even, then the smallest subspace among $U^{n+1},U^{n-1},\ldots,U^{1}$ which intersects $C$ is $U^1$. Hence the westernmost point of $C$ is the north pole $C\cap U^1 = \{(0,\ldots, 0,1)\}$.
If $n$ is odd, then the smallest subspace which intersects $C$ is $U^2$.
Hence the westernmost point of $C$ is the optimizer of $\theta_1$ over the half-circle $\mathbb{S}^n_{\geq0}\cap U^2\cong \mathbb{S}^1_{\geq0}$, which is $\{(0,\ldots,0,1,0)\}$.

\item If $P$ is not full dimensional, %
not every cell of $\mathcal G(P)$ needs to have a westernmost point:
Consider $P=\operatorname{conv}\{v,-v\}\subseteq \RR^3$ for a generic $v\in\mathbb{S}^2$; see Figure~\ref{fig:westernmostCorners} (b). Then $\mathcal G(P)$ consists of %
two open half-spheres and a great circle. Each of the half-spheres intersects $U^1$ at a single point, which is their westernmost point. The great circle %
does not intersect $U^1$ and
has no local optimum for the direction west. Hence it %
has no westernmost point. Note however that each %
cell of $\mathcal G^+(P)$ has a unique westernmost point; see Figure~\ref{fig:westernmostCorners}~(c).
\item If a cell is not in general orientation relative to the $U^i$, then it can have multiple westernmost points: Consider a segment of a great circle in $\mathbb{S}^2$ passing through north and south poles. If it does not contain either pole, then any of its points is a westernmost point.
\end{enumerate}
\end{example}

\begin{lemma}\label{lemma:existence}
  Let $P\subseteq \RR^{n+1}$ be a polytope. If $P$ is in general orientation (relative to the coordinate system), then each cell of $\mathcal G^+(P)$ has a unique westernmost corner. Additionally, if $P$ is full-dimensional, then each cell of $\mathcal G(P)$ has a unique westernmost corner.
\end{lemma}
\begin{proof}
Note that cells of $\mathcal G^+(P)$ and, for $P$ full-dimensional, of $\mathcal G(P)$ are spherically convex in the sense that any shortest arc between two points is inside the set. Combined with $P$ being in general orientation, \cite[Lemma 6]{Weibel12} shows that every cell has a westernmost point and %
corner. %
Since $P$ is in general orientation, each cell has a single westernmost point.
\end{proof}

\begin{definition}[Support]
Let $P_1,\ldots,P_m\subseteq\RR^{n+1}$ be a family of polytopes. The support of a point $w\in\mathbb{S}^n$ is defined as
\[ \Supp_{P_1,\dots,P_m}(w)\coloneqq\{ i\in [m]\colon  w
\in C_i \text{ for some } C_i\in \mathcal{G}(P_i) \text{ with }\operatorname{co-dim}(C_i)\geq 1\}\subseteq [m]. \]
In particular, the support of a generic point is $\emptyset$.
The support $\Supp_{P_1,\dots,P_m}(W)$ of
a westernmost corner $W$ of a cell %
is defined to be the support of its westernmost point.
\end{definition}

The following lemma points out that for polytopes in general orientation, westernmost corners that appear in one sub-complex also appear in any larger sub-complex.

\begin{lemma}\label{lemma:sequence}
  Let $P_1,\ldots, P_m$ be a family of polytopes in general orientation, and let $W$ be a westernmost corner of an $s$-cell of $\mathcal G(P_{[m]})$. Then $W$ is the westernmost corner of an $s$-cell of $\mathcal G(P_S)$ if and only if $\Supp_{P_1,\ldots, P_m}(W)\subseteq S$.
\end{lemma}
\begin{proof}
  Note that $\mathcal G(P_S)$ coincides with $\mathcal G(P_{[m]})$ locally around $W$ if and only if its support $\Supp_{P_1,\ldots, P_m}(W)$ is contained in $S$. Thus the claim is an immediate consequence of \cite[Lemma~7]{Weibel12}, which states that for spherically convex cells on $\mathbb S^n$ in general orientation westernmost points are %
  the local optima of the direction west.
\end{proof}

We will use the following lemma in order to enumerate the westernmost points of Minkowski sums based on the above observation.
Notice that ${m-r \choose j-r}$ is the number of subsets of $[m]$ of cardinality $j$ which contain some particular subset of $[m]$ of cardinality $r$.

\begin{lemma}\label{lemma:inclusion-exclusion}
 For any integers $0\leq r\leq n<m$, we have $\sum_{j=0}^n (-1)^{n-j} {m-1-j\choose n-j} {m-r \choose j-r} = 1$.
 In particular, for any function $\xi_{(\cdot)}:2^{[m]}\rightarrow\RR$ with $\sum_{S\in\binom{[m]}{j}} \xi_{S}={m-r \choose j-r}$ for all $0\leq j\leq n$, %
 \[ \sum_{j=0}^n (-1)^{n-j} {m-1-j\choose n-j} \sum_{S\in\binom{[m]}{j}} \xi_{S}= 1. \]
\end{lemma}

\begin{proof}
 The proof follows by induction over $m\geq n+1$, using for $m=n+1$ the fact that $\sum_{j=0}^{n+1} (-1)^j{n+1-r\choose j-r}=0$ and hence $\sum_{j=0}^n (-1)^{n-j}{n+1-r\choose j-r}=1$ for any $0\leq r<n+1$.
\end{proof}

We obtain the following linear relation between the number of upper $s$-faces of a Minkowski sum of $m$ polytopes and the number of upper $s$-faces of subsums of at most $n$ of the polytopes.
This is a version of Weibel's theorem \cite[Theorem~1]{Weibel12} for the case of upper faces. Whereas that result is for sums of full-dimensional polytopes, our statement is valid for any dimensions. %

\begin{theorem}[Number of upper faces of Minkowski sums]
\label{thm:f_vectors_upper_part_Minkowski_sum}
Let $P_1,\ldots, P_{m}$ be %
any positive dimensional polytopes in $\mathbb{R}^{n+1}$ in general orientations, $m\geq n+1$, and $P=P_1+\cdots+P_m$.
Then for the number of $s$-faces of the upper part we have
$$
f_s(P^+) = \sum_{j=0}^{n} (-1)^{n-j} {m-1-j \choose n-j} \sum_{S\in{[m]\choose j}} f_s(P_S^+) , \quad s=0,\ldots, n,
$$
where
$P_S = (\sum_{i\in S}P_i)$ for any nonempty $S\subseteq[m]$, and $P_\emptyset=\{0\}$.
\end{theorem}

\begin{proof}
  Consider the %
  complex $\mathcal G^+(P)$, and recall that $s$-dimensional upper faces of $P$ correspond to $(n-s)$-dimensional cells of $\mathcal G^+(P)$.
  Let $W_1,\dots,W_N$ be the westernmost corners of $(n-s)$-cells of $\mathcal G^+(P)$ and let $I_1,\dots,I_N\subseteq [m]$ denote their supports.
  Note that %
  $0\leq |I_i| \leq n$ for all $i=1,\dots,N$. Let $w_s(P_S^+)$ denote the number of west-most corners of $(n-s)$-cells of $P_S^+$, so that $w_s(P_S^+)=f_s(P_S^+)$. Writing $f_s(P^+) = N = \sum_{i=1}^N 1$, we then obtain
  \begin{align*}
    &f_s(P^+)\overset{Lem.}{\underset{\ref{lemma:inclusion-exclusion}}{=}}\sum_{i=1}^N \sum_{j=0}^n (-1)^{n-j} \binom{m-1-j}{n-j}\!\!\!\! \sum_{S\in {[m] \choose j}}\!\!\! \mathbb{1}_{I_i\subseteq S}
    = \sum_{j=0}^n (-1)^{n-j} \binom{m-1-j}{n-j}\!\!\!\! \sum_{S\in {[m] \choose j}} \sum_{i=1}^N \mathbb{1}_{I_i\subseteq S} \\
    &\overset{Lem.}{\underset{\ref{lemma:sequence}}{=}} \sum_{j=0}^n (-1)^{n-j} \binom{m-1-j}{n-j} \!\!\sum_{S\in {[m] \choose j}} w_s(P_S^+)
    = \sum_{j=0}^n (-1)^{n-j} \binom{m-1-j}{n-j} \!\!\sum_{S\in {[m] \choose j}} f_s(P_S^+). \qedhere %
  \end{align*}
\end{proof}
One naturally wonders if the proof of Theorem~\ref{thm:f_vectors_upper_part_Minkowski_sum} can be extended to count the faces of the entire polytope, generalizing Weibel's result to sums of polytopes of arbitrary dimensions.
Lemma~\ref{lemma:existence} does not cover that case.
We will present an alternative approach in Section~\ref{sec:Weibel-Zaslavsky}.

Weibel \cite[Theorem~3]{Weibel12} also shows, for sums of full-dimensional polytopes, that the number of vertices is maximized when the partial sums attain the trivial upper bound.
The same arguments transfer to the case of upper vertices, and one can show the following corollary.
In the following we consider families of polytopes $P_i$ satisfying $f_0(P_i)= k_i$, $i=1,\ldots, m$.
\begin{corollary}[Upper bound for upper vertices of sums of full-dimensional polytopes]
Let $m\geq n+1$ and $k_1,\ldots, k_m\geq n+2$.
Then
$f_0(P^+) \leq \sum_{j=0}^n (-1)^{n-j}{m-1-j\choose n-j}\sum_{S\in{[m]\choose j}}\prod_{i\in S} k_i$.
\end{corollary}
By comparison, Weibel's upper bound for the total number of vertices
is $f_0(P) \leq \binom{m-1}{n} +$ \linebreak $\sum_{j=0}^n (-1)^{n-j}{m-1-j\choose n-j}\sum_{S\in{[m]\choose j}}\prod_{i\in S} k_i$.
Unfortunately his argument does not extend to the case where some of the $k_i$ are small, neither for all vertices nor for upper vertices.
To address that case, we will instead use the upper bound theorem by Adiprasito-Sanyal~\cite{KarimRaman}. Their result implies Proposition~\ref{thm:Mneighborly_upper_bound}, which states that if a Minkowski sum maximizes the number of vertices, then the partial sums attain the trivial upper bound $f_0(P_S)=\prod_{i\in S}f_0(P_i)$.
The problem remaining is whether maximizing the number of vertices $f_0(P)$ also entails maximizing the number of upper vertices $f_0(P^+)$ and whether the partial sums will also attain the trivial upper bound for upper vertices.
In Section~\ref{sec:Weibel-Zaslavsky} we will show that this is indeed the case.

We find it useful to rewrite the alternating sum as follows.
\begin{lemma}
\label{lemma:reformulation}
Let %
$m\ge n+1\geq 1$ %
and $k_1,\ldots, k_m \geq 2$. %
Then
\[\sum_{j=0}^{n} (-1)^{n-j} {m-1-j \choose n-j} \sum_{S\in{[m]\choose j}}\prod_{i\in S} k_i = \sum_{j=0}^n\sum_{S\in{[m]\choose j}}\prod_{i\in S}(k_i-1).\]
\end{lemma}
\begin{proof}
We prove the %
equality by viewing both sides as polynomials in the variables $k_i$ and examining the coefficient %
of each monomial. %
Fix a subset $S\subseteq [m]$ of size $j$. The coefficient %
for the monomial $k_S=\prod_{i\in S}k_i$ on the left-hand side of the equation is $(-1)^{n-j}\binom{m-1-j}{n-j}$.
On the right-hand side, the term $k_S$ appears with sign equal to $(-1)^{|T|-|S|}$ for each $T\supseteq S$. The coefficient %
on the right-hand side is therefore
$
\sum_{T\supseteq S} (-1)^{|T|-j}=\sum_{i = 0}^{n-j} (-1)^i \cdot \binom{m-j}{i}$.
The statement now follows %
from the following observation, which is obtained by induction on $n$:
 If $m \geq n+1\geq 1$, then $\sum_{i=0}^{n} (-1)^i  \binom{m}{i} = (-1)^{n}\binom{m-1}{n}$.
\end{proof}

\section{Face counting formulas \`a la Zaslavsky}
\label{sec:Zaslavsky}

Zaslavsky \cite{zaslavsky1975facing} proved a theorem expressing the number of regions defined by a hyperplane arrangement in terms of the characteristic polynomial, a function obtained from the intersection poset of the arrangement.
In the following we derive a similar type of result for the case of maxout arrangements.
Hyperplanes are special in that their intersections are affine spaces and can be discussed in terms of linear independence relations.
In contrast, for maxout arrangements the intersections involve linear equations and also linear inequalities.
In turn, the elements of the poset have a more complex topology that needs to be accounted for. In general the poset also has a more complex structure even if the arrangement is in general position.

\begin{definition}[Maxout arrangement]
For a collection of
$m$ maxout units %
$z_i(x) =$\linebreak $\max\{A_{i,1}(x),\ldots, A_{i,k_i}(x)\}$, $x\in\mathbb{R}^n$, $i=1,\ldots, m$, we define the maxout arrangement $\mathcal{A}=\{H_{ab}^i\colon
 \{a,b\}\in {[k_i]\choose 2}, i\in[m],  \operatorname{co-dim}(H^i_{ab})=1\}$ in $\mathbb{R}^n$ as the collection of nonempty co-dimension $1$ indecision boundaries between pairs of preactivation features, called atoms,
\begin{equation}
H_{ab}^i = \left\{x\in\mathbb{R}^n\colon A_{i,a}(x) = A_{i,b}(x) =\max_{c\in[k_i]}\{A_{i,c}(x)\} \right\}.
 \label{eq:maxoutarrangementpieces}
\end{equation}
We call the arrangement central if the affine functions $A_{i,a}$ of each unit are linear.

We let $L(\mathcal{A})$ denote the set of all possible nonempty sets obtained by intersecting subsets of elements in $\mathcal{A}$, including $\mathbb{R}^n$ as the empty set intersection. The set $L(\mathcal{A})$ is partially ordered by reverse inclusion, so that for any $s,t\in L(\mathcal{A})$ we have $s\geq t$ if and only if $s\subseteq t$. The smallest element, i.e.\ the $\hat 0$ in this poset, is $\mathbb{R}^n$.
For a given arrangement $\mathcal{A}$, we denote by $r(\mathcal{A})$ the number of connected components of $\mathbb{R}^n\setminus \cup_{H\in\mathcal{A}}H$, called the regions of $\mathcal{A}$.
\end{definition}
Note that rank-$1$ units have no indecision boundaries and can be ignored. In the following we will therefore assume without loss of generality that $k_1,\ldots, k_m\geq 2$.

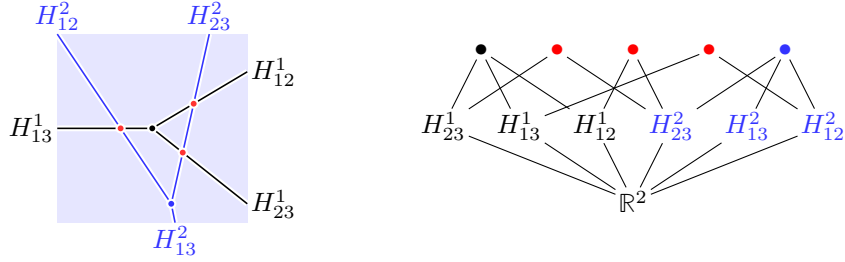
\begin{figure}
    \centering
\begin{tabularx}{115mm}{m{5cm}m{65mm}}
\begin{tikzpicture}[every node/.style={black,above right, inner sep=1pt}]

\path[fill=blue!10] (-1.25,-1.25) rectangle (1.25cm,1.25cm);

\draw[name path=line11, double=black, white, thick] (0,0) -- (1.25,.75) node [right] {$H^1_{12}$};
\draw[name path=line12, double=black, white, thick] (0,0) -- (1.25,-1) node [right] {$H^1_{23}$};
\draw[name path=line13, double=black, white, thick] (0,0) -- (-1.25,0) node [left] {$H^1_{13}$};

\draw[name path=line21, double=blue!80, white, thick] (.25,-1) -- (-1.25,1.25) node [above] {\textcolor{blue!80}{$H^2_{12}$}};
\draw[name path=line22, double=blue!80, white, thick] (.25,-1) -- (.75,1.25) node [above] {\textcolor{blue!80}{$H^2_{23}$}};
\draw[name path=line23, double=blue!80, white, thick] (.25,-1) -- (.3,-1.25) node [below] {\textcolor{blue!80}{$H^2_{13}$}};

\fill[name intersections={of=line11 and line12,total=\t}, draw=white, thick] {(intersection-1) circle (1.5pt) node {}};
\fill[name intersections={of=line21 and line22,total=\t}, blue!80, draw=white, thick] {(intersection-1) circle (1.5pt) node {}};

\foreach \i in {1,2,3}{
	\foreach \j in {1,2,3}{
       \fill[name intersections={of={line2\i} and {line1\j}, total=\t}, red!80, draw=white, thick][]
           \ifnum\t=0
           {};
           \else
        	\foreach \s in {1,...,\t}{(intersection-\s) circle (1.5pt) node {} } ;
         \fi
   	}
}
\end{tikzpicture}
&
\begin{tikzpicture}[inner sep=1pt]
\node (zero) at (0,-1) {$\mathbb{R}^2$};
\node (H112) at (-.5,0) {$H^1_{12}$};
\node (H113) at (-1.5,0) {$H^1_{13}$};
\node (H123) at (-2.5,0) {$H^1_{23}$};

\node (H212) at (2.5,0) {\textcolor{blue!80}{$H^2_{12}$}};
\node (H213) at (1.5,0) {\textcolor{blue!80}{$H^2_{13}$}};
\node (H223) at (.5,0) {\textcolor{blue!80}{$H^2_{23}$}};

\node (H10) at (-2,1) {\textcolor{black}{$\bullet$}}; %
\node (H20) at (2,1) {\textcolor{blue!80}{$\bullet$}}; %

\node (H113-H212) at (1,1) {\textcolor{red}{$\bullet$}};%
\node (H123-H223) at (-1,1) {\textcolor{red}{$\bullet$}};%
\node (H112-H223) at (0,1) {\textcolor{red}{$\bullet$}};%

\draw (zero) -- (H112) -- (H10) -- (H113) -- (zero) -- (H123) -- (H10);
\draw (zero) -- (H212) -- (H20) -- (H213) -- (zero) -- (H223) -- (H20);
\draw (H113) -- (H113-H212) -- (H212);
\draw (H123) -- (H123-H223) -- (H223);
\draw (H112) -- (H112-H223) -- (H223);
\end{tikzpicture}
\end{tabularx}\vspace{-3mm}
\caption{Shown is an arrangement of two maxout units of ranks $k_1=k_2=3$ on $\mathbb{R}^2$ along with its intersection poset discussed in Example~\ref{ex:0} and \ref{ex:1}.
}
    \label{fig:posetcounting}
\end{figure}

\begin{example}\label{ex:0}
  Figure~\ref{fig:posetcounting} shows a maxout arrangement of two rank-$3$ maxout units and their intersection poset. The black indecision boundaries arise from $z_1(x,y)=\max\{2y,x+y+1,2\}$, while the blue indecision boundaries arise from $z_2(x,y)=\max\{0,3x+2y,5x+y\}$.
\end{example}

Recall that the M\"obius function of a poset $L$ with partial order $\leq$ is defined by $\mu_L(s,s) = 1$ for $s\in L$,
$\mu_L(s,u) =  -\sum_{s\leq t< u} \mu_L(s,t)$ for $s<u\in L$,
and $\mu_L(s,u) = 0$ for $s\not\leq u$.
If the poset has a minimal element $\hat 0$, one also defines $\mu_{L}(x):=\mu_{L}(\hat 0,x)$.
The M\"obius inversion formula for locally finite posets states %
the following %
equivalence for functions $g$ and $h$ on the poset \cite{Rota}:
\begin{align*}
g(t) = \sum_{s\leq t} h(s) \quad \forall t\in L \quad\text{if and only if}\quad
h(t) = \sum_{s\leq t} g(s) \mu_L(s,t) \quad \forall t\in L.
\end{align*}

Further, recall that if a space $X$ is suitably decomposed into cells with $f_s$ cells of dimension $s$, then its Euler characteristic is defined as $\psi(X) := f_0 - f_1 + f_2 \mp \cdots$, whereby we follow the notation of \cite{Stanley04anintroduction}.
Concretely, a polyhedral decomposition or a CW complex decomposition with finitely many pieces are suitable for computing the Euler characteristic; see \cite[Chapter 4 Proposition 2.2]{1998tame}.
The Euler characteristic is independent of the specific decomposition.

A (closed) face of the arrangement $\mathcal{A}$ is a set of the form $\emptyset\neq F = \overline{R}\cap x$, where $\overline{R}$ is the closure of a connected component $R$ of $X\setminus\cup_{H\in\A}H$, and $x\in L(\mathcal{A})$. Denote the set of faces of $\mathcal{A}$ by $\mathcal{F}(\mathcal{A})$.
The faces of an arrangement $\mathcal{A}$ in $X$ create a decomposition $X = \sqcup_{F\in F(\mathcal{A})}\operatorname{relint}(F)$.

\begin{definition}[Proper arrangement]
\label{def:properArrangement}
We call an arrangement $\A$ in a space $X$ proper if the set of faces $\mathcal{F}(\mathcal{A})$ decomposes $X$ suitably for computing the Euler characteristic.
\end{definition}

\begin{lemma}\label{lem:properArrangements}
A maxout arrangement in $\mathbb{R}^n$ is always proper.
A central maxout arrangement in $\mathbb{R}^{n+1}$ restricted to $\mathbb{S}^n$ is proper whenever the associated Newton polytope in $\mathbb{R}^{n+1}$ has dimension $n+1$.
\end{lemma}
\begin{proof}
The maxout arrangement creates a polyhedral decomposition of $\mathbb{R}^n$.
For a central arrangement with full-dimensional polytope, the spherical cell complex is equivalent to the boundary of the dual polytope, which is a CW complex.
\end{proof}

We obtain the following counting formula for the number of regions of an arrangement.

\begin{theorem}[Number of faces of maxout arrangements]
\label{thm:posetcounting}
Consider a proper arrangement $\mathcal{A}$ in a space $X$ (e.g.\ $\mathbb{R}^n$ or $\mathbb{S}^n$). Then the number of regions defined by $\mathcal{A}$ on $X$ satisfies
$$
r(\mathcal{A}) = (-1)^{\dim(X)} \sum_{y\in L(\mathcal{A})} \psi(y) \mu_{L(\mathcal{A})}(y).
$$
Moreover, writing $L(\A)/x = \{y\in L(\A) \colon y\geq x\}$, the number of $s$-faces satisfies
$$
f_s(\A)
= \sum_{\substack{x\in L(\A)\\\dim(x)=s}}(-1)^{\dim(x)}\sum_{y\in L(\A)/x} \psi(y)\mu_{L(\A)/x}(y)
, \quad s=0,\ldots, \dim(X)-1.
$$
\end{theorem}
This result is an instance of what Zaslavsky calls a fundamental theorem of dissection theory \cite[Theorem~1.2]{ZASLAVSKY1977267}.
We note that if $\mathcal{A}$ is an arrangement of hyperplanes in $\mathbb{R}^n$, then each $y\in L(\mathcal{A})$ is an affine subspace with Euler characteristic $\psi(y) = (-1)^{\dim(y)}$ and the statement of Theorem~\ref{thm:posetcounting} corresponds to Zaslavsky's result \cite[Theorem~A]{zaslavsky1975facing}, which states that for hyperplane arrangements
$r(\mathcal{A})
=(-1)^n \sum_{x\in L(\mathcal{A})}(-1)^{\dim(x)} \mu_{L(\mathcal{A})}(x) =(-1)^n \chi_{L(\mathcal{A})}(-1)$. Here, the characteristic polynomial of $L(\mathcal{A})$ is defined as $\chi_{L(\mathcal{A})}(t) := \sum_{x\in L(\mathcal{A})} \mu_{L(\mathcal{A})}(x) t^{\dim(x)}$, $t\in\mathbb{R}$. %

\begin{proof}[Proof of Theorem~\ref{thm:posetcounting}]
The proof follows the Eulerian method, i.e.\ the arguments of Zaslavsky's quick proofs in \cite{zaslavsky1975facing}.
Consider an arrangement $\mathcal{A}$ in a space $X$.
For any $y\in L(\mathcal{A})$, denote the arrangement induced by $\mathcal{A}$ on $y$ as
$\mathcal{A}^y := \{y\cap H\neq \emptyset \colon H\in \mathcal{A}, H\not\supseteq y\}$.
Now, every $k$-face of $\mathcal{A}$ is %
closure of a region of exactly one $\mathcal{A}^y$, $y\in L(\mathcal{A})$, $\dim(y)=k$.
Hence we have
\begin{equation}
f_k(\mathcal{A}) = \sum_{\substack{y\in L(\mathcal{A})\\ \dim(y)=k}}r(\mathcal{A}^y).
\label{eq:kfaces}
\end{equation}
Hence, if $\A$ %
is proper, then we have
$\psi(X) = f_0(\mathcal{A}) - f_1(\mathcal{A}) \pm %
\cdots
= \sum_{y\in L(\mathcal{A})} (-1)^{\dim(y)} r(\mathcal{A}^y)$.
We have an analogous expression for any $x\in L(\mathcal{A})$ in place of $X$.
Note that if $\mathcal{A}$ is proper, then so is $\mathcal{A}^y$ for any $y\in L(\mathcal{A})$.
Hence,
$$
\psi(x)%
= \sum_{\substack{y\in L(\mathcal{A})\\ y\geq x}} (-1)^{\dim(y)} r(\mathcal{A}^y)\quad \forall x\in L(\mathcal{A}).
$$
The M\"obius inversion formula then gives
$$
(-1)^{\dim(x)} r(\mathcal{A}^x) = \sum_{\substack{y\in L(\mathcal{A})\\ y\geq x}} \psi(y) %
\mu_{L(\mathcal{A})}(x,y) \quad \forall x\in L(\mathcal{A}).
$$
Substituting $x=X$ gives
$(-1)^{\dim(X)} r(\mathcal{A}) = \sum_{y\in L(\mathcal{A})} \psi(y) %
\mu_{L(\mathcal{A})}(y)$.
This completes the proof of the first statement.
The second statement is by \eqref{eq:kfaces} and applying the first statement to each element of the intersection poset of dimension $s$.
\end{proof}

\begin{example}
\label{ex:1}
We illustrate Theorem~\ref{thm:posetcounting}.
We consider the arrangement of two maxout units of ranks $k_1=k_2=3$ in $\mathbb{R}^n$, $n=2$, shown in Figure~\ref{fig:posetcounting}.
In this example
$\mu(\mathbb{R}^2)=1$,
$\mu(H^i_{ab})=-1$,
$\mu(\bullet)=2$,
$\mu(\textcolor{blue!80}{\bullet})=2$,
$\mu(\textcolor{red}{\bullet})=1$,
and
$\psi(\mathbb{R}^2)=(-1)^2=1$,
$\psi(H^i_{ab})=1-1 = 0$,
$\psi(\bullet)= \psi(\textcolor{blue!80}{\bullet})= \psi(\textcolor{red}{\bullet})=(-1)^0=1$.
By Theorem~\ref{thm:posetcounting}, the number of regions and $1$-faces are
\begin{align*}
r(\mathcal{A}) & = \sum_{y\in L(\mathcal{A})} \psi(y)\mu_{L(\mathcal{A})}(y) = 1 + 0 (-1\!-\!1\!-\!1\!-\!1\!-\!1\!-\!1) + (-1)^0 (2\!+\!2\!+\!1\!+\!1\!+\!1) = 8 ,\\
f_1(\A) & = \sum_{\substack{x\in L(\A)\\ \dim(x)=1}}(-1)^{\dim(x)}\sum_{y\in L(\A)/x} \psi(y)\mu_{L(\A)/x}(y)\\ & =(-1)^1(\underbracket[0.5pt]{(0\!-\!1 \!-\!1)}_{x=H^1_{12}} + \underbracket[0.5pt]{(0\!-\!1\!-\!1)}_{x=H^1_{13}} + \underbracket[0.5pt]{(0\!-\!1\!-\!1)}_{x=H^1_{23}} + \underbracket[0.5pt]{(0\!-\!1\!-\!1)}_{x=H^2_{12}} + \underbracket[0.5pt]{(0\!-\!1)}_{x=H^2_{13}} + \underbracket[0.5pt]{(0\!-\!1\!-\!1\!-\!1)}_{x=H^2_{23}}) = 12 .
\end{align*}
\end{example}

For a general maxout arrangement, the Euler characteristic of the elements of the intersection poset may depend not only on the dimension.
Besides the case of hyperplane arrangements, we note another special case where it depends only on the dimension.
For central arrangements of full-dimensional maxout units in $\mathbb{R}^{n+1}$ (i.e.\ each of the polytopes $P_1,\ldots, P_m$ has dimension $n+1$), each $y\in L(\A)\setminus\{\mathbb{R}^{n+1}, \{0\}\}$ is an unbounded pointed cone with $\psi(y)=0$ (this can be deduced from the fact that any bounded convex polytope has Euler characteristic $1$).
Hence, noting that $\psi(\mathbb{R}^{n+1})=(-1)^{n+1}$ and $\psi(\{0\})=1$, we may express the number of regions of such an $\A$ in terms of the characteristic polynomial
$\chi_{L(\A)}(t) = \sum_{y\in L(\A)} \mu_{L(\A)}(y) t^{\dim(y)}$ as follows.
\begin{proposition}
For a central arrangement $\A$ of full-dimensional maxout units in $\mathbb{R}^{n+1}$,
$r(\A)
= 1 + (-1)^{n+1}  \chi_{L(\A)}(0)$.
\end{proposition}

\section{Region counting by regions of sub-arrangements}
\label{sec:Weibel-Zaslavsky}

Our objective in this section is to write %
the number of regions $r(\mathcal{A})$ of an arrangement $\mathcal{A}$
in terms of the numbers of regions of small sub-arrangements.
To this end we will focus on simple arrangements, which arise from maxout units with generic parameters, and combine the results from Sections~\ref{sec:Weibel} and \ref{sec:Zaslavsky}.
As before, we will assume that each maxout unit has rank at least two.
To draw a bridge between the upper vertices and all vertices of Minkowski sums of polytopes, we consider both non-central and central arrangements.

\begin{definition}[Simple arrangements]
    Let $\mathcal{A}=\{H^{i}_{ab}\}_{i\in [m],\{a,b\}}$ be a maxout arrangement.
    For any $S\subseteq[m]$, write $\mathcal{A}_S=\{H^{i}_{ab}\}_{i\in S,\{a,b\}}$ for the sub-arrangement of atoms of units $i\in S$.

  A maxout arrangement $\A$ is simple, or in general position, if the intersection of any $j$ atoms of different units is either empty or has co-dimension $j$.
In the case of a central arrangement, all atoms contain the origin $0$.
A central maxout arrangement $\A$ is simple if the intersection of any $j$ atoms of different units is either the origin or has co-dimension $j$.

If $\A$ is simple, we define the support of an element $y=\cap_{l=1}^r H^{i_l}_{a_l b_l}\in L(\mathcal A)$, with $y\neq 0$ if $\A$ is central,  as the set $\{i_1,\dots,i_r\}\subseteq [m]$.
Note that while $y$ may not have a unique representation as intersection of atoms, the support is well-defined if $\A$ is simple and $y\neq 0$ if $\A$ central.
\end{definition}

\subsection{Simple non-central arrangements} %
Simple non-central arrangements have the convenient property that %
an element $y\in L(\A)$ of the intersection poset is %
contained in the intersection poset $L(\mathcal A_S)$ of a sub-arrangement $\mathcal{A}_S$ if and only if the support of $y$ is contained in $S$.
This mimics the behavior of westernmost corners in Lemma~\ref{lemma:sequence}, but without the need of a coordinate system and local optimization.
We obtain a formula similar to Theorem~\ref{thm:f_vectors_upper_part_Minkowski_sum} for upper vertices with a short proof based on Theorem~\ref{thm:posetcounting}.

\begin{theorem}[Number of regions of a simple non-central arrangement]
\label{thm:facessimple}
Let $\A$ be a simple arrangement of $m\geq n+1$ maxout units of ranks $k_1,\ldots, k_m\geq 2$ in $\mathbb{R}^n$. Then %
\begin{align*}
r(\mathcal{A}) = \sum_{j=0}^{n} (-1)^{n-j}{m-1-j\choose n-j} \sum_{S\in {[m]\choose j}}  r(\mathcal{A}_S).
\end{align*}
\end{theorem}

\begin{proof}[Proof of Theorem~\ref{thm:facessimple}]
If $\A$ is simple, then any element in $L(\mathcal{A})$ can be written as an intersection of at most $n$ atoms.
Moreover, any $y\in L(\A)$ is included in $L(\mathcal A_S)$ if and only if its support is contained in $S$.
Hence, using Lemma~\ref{lemma:inclusion-exclusion} %
and the fact that $\mu_{L(\mathcal{A})}(y) = \mu_{L(\mathcal{A}_S)}(y)$ for any $y\in L(\mathcal{A}_S)$, we can rewrite the expression from Theorem~\ref{thm:posetcounting} as follows:
\allowdisplaybreaks
\begin{align}\label{eq:nonCentralRegions}
    r(\mathcal{A}) &\textoverset[O]{Thm.~\ref{thm:posetcounting}}{=} (-1)^n \sum_{y\in L(\mathcal{A})}\psi(y)\mu_{L(\mathcal{A})}(y)\nonumber \\
    &\textoverset[O]{Lem.~\ref{lemma:inclusion-exclusion}}{=} (-1)^n \sum_{y\in L(\mathcal{A})}\psi(y)\mu_{L(\mathcal{A})}(y) \sum_{j=0}^n (-1)^{n-j} \binom{m-1-j}{n-j} \sum_{S\in {[m] \choose j}} \mathbb{1}_{y\in L(\A_S)}\nonumber \\
    &\textoverset[O]{}{=} (-1)^n \sum_{j=0}^{n} (-1)^{n-j}{m-1-j\choose n-j} \sum_{S\in {[m]\choose j}}  \sum_{y\in L(\mathcal{A}_S)} \psi(y)\mu_{L(\mathcal{A}_S)}(y) \\
    &\textoverset[O]{Thm.~\ref{thm:posetcounting}}{=} \sum_{j=0}^{n} (-1)^{n-j}{m-1-j\choose n-j} \sum_{S\in {[m]\choose j}}  r(\mathcal{A}_S).\nonumber \qedhere
\end{align}
\end{proof}

\subsection{Simple central arrangements}
Simple central arrangements come with the added challenge that the origin $\{0\}= \hat 1\in L(\mathcal{A})$ may be contained in several posets $L(\mathcal{A}_S)$ with $S\subseteq [m]$, $S\neq \emptyset$, while not being contained in the poset corresponding to the intersection of these $S$.
See Figure~\ref{fig:originInCentralArrangements}.
Consequently, there is no obvious way to define a support for $\{0\}$ that is consistent with %
Lemma~\ref{lemma:sequence}, which in turn makes it difficult to apply Lemma~\ref{lemma:inclusion-exclusion}. We thus need %
to treat $\{0\}$ separately.
We consider central arrangements in $\RR^{n+1}$, corresponding to the ambient space of the lifted Newton polytopes of maxout networks with input space~$\RR^n$.

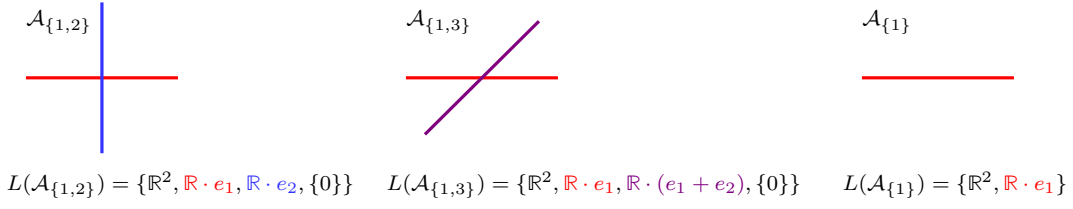
\begin{figure}
    \centering
    \begin{tikzpicture}[font=\footnotesize]
      \node (A12) at (0,0)
      {%
        \begin{tikzpicture}
          \useasboundingbox (-1,-1) rectangle (1,1);
          \draw[very thick,red] (-1,0) -- (1,0);
          \draw[very thick,blue!80] (0,-1) -- (0,1);
          \node[anchor=north west, font=\footnotesize] at (-1.1,1) {$\A_{\{1,2\}}$};
        \end{tikzpicture}
      };
      \node[anchor=north west,xshift=-2.5mm] (A12Text) at (A12.south west)
      {$L(\A_{\{1,2\}})=\{\RR^2,\textcolor{red}{\RR\cdot e_1},\textcolor{blue!80}{\RR\cdot e_2},\{0\}\}$};
      \node (A13) at (5,0)
      {%
        \begin{tikzpicture}
          \useasboundingbox (-1,-1) rectangle (1,1);
          \draw[very thick,red] (-1,0) -- (1,0);
          \draw[very thick,violet] (-0.75,-0.75) -- (0.75,0.75);
          \node[anchor=north west, font=\footnotesize] at (-1.1,1) {$\A_{\{1,3\}}$};
        \end{tikzpicture}
      };
      \node[anchor=north west,xshift=-2.5mm] (A13Text) at (A13.south west)
      {$L(\A_{\{1,3\}})=\{\RR^2,\textcolor{red}{\RR\cdot e_1},\textcolor{violet}{\RR\cdot (e_1+e_2)},\{0\}\}$};
      \node (A1) at (11,0)
      {%
        \begin{tikzpicture}
          \useasboundingbox (-1,-1) rectangle (1,1);
          \draw[very thick,red] (-1,0) -- (1,0);
          \node[anchor=north west, font=\footnotesize] at (-1.1,1) {$\A_{\{1\}}$};
        \end{tikzpicture}
      };
      \node[anchor=north west,xshift=-2.5mm] (A1Text) at (A1.south west)
      {$L(\A_{\{1\}})=\{\RR^2,\textcolor{red}{\RR\cdot e_1}\}$};
    \end{tikzpicture}\vspace{-3mm}
    \caption{The origin $\{0\}$ is contained in both the intersection poset of $\A_{\{1,2\}}$ and of $\A_{\{1,3\}}$, but not in the intersection poset of $\A_{\{1,2\}\cap\{1,3\}}$.
    }
    \label{fig:originInCentralArrangements}
\end{figure}
\bigskip

\paragraph{Special case}
First we present a simple approach to enumerate the regions of a central arrangement in the special case that the $\mathcal{A}_{\{i\}}$ are proper. We consider the arrangement $\A\cap\mathbb{S}^n$, which for simplicity of notation we will still denote $\A$. For a central arrangement in $\mathbb{R}^{n+1}$, the number of regions it defines in $\mathbb{R}^{n+1}$ is equal to the number of regions it defines on $\mathbb{S}^n$.
Conveniently, $\{0\}$ no longer appears in the posets over $\mathbb{S}^n$. For a central arrangement of $m\geq n+1$ 
units of ranks $k_1,\ldots, k_m\geq 2$, as before in~\eqref{eq:nonCentralRegions},
\begin{align}
r(\mathcal{A}) = (-1)^n\!\!\sum_{y\in L(\mathcal{A})}\!\!\!\psi(y)\mu_{L(\mathcal{A})}(y)
= & (-1)^n \sum_{j=0}^{n} (-1)^{n-j}{m-1-j\choose n-j} \!\!\sum_{S\in {[m]\choose j}}\sum_{y\in L(\mathcal{A}_S)}\!\!\!\psi(y)\mu_{L(\mathcal{A}_S)}(y) .
\label{eq:centralsphere}
\end{align}
Under the additional assumption that the arrangement $\mathcal{A}_S$ on $\mathbb{S}^n$ is proper for all $S\neq\emptyset$ (e.g.\ 
the lifted Newton polytopes of all maxout units are full-dimensional), we can use Theorem~\ref{thm:posetcounting} to rewrite \eqref{eq:centralsphere}~as
 \begin{align}
r(\mathcal{A})
 = & (-1)^n{m-1\choose n} \psi(\mathbb{S}^n)  + \sum_{j=1}^{n} (-1)^{n-j}{m-1-j\choose n-j} \sum_{S\in {[m]\choose j}}  r(\mathcal{A}_S)\nonumber\\
= &  \psi(\mathbb{S}^n) + \sum_{j=1}^{n} (-1)^{n-j}{m-1-j\choose n-j} \sum_{S\in {[m]\choose j}}  ( r(\mathcal{A}_S) - \psi(\mathbb{S}^n) ),
\label{eq:centralsphererecoverWeibel}
 \end{align}
in which we also use $\sum_{y\in L(\mathcal{A}_\emptyset)}\psi(y)\mu_{L(\mathcal{A}_\emptyset)}(y) %
=\psi(\mathbb{S}^n)$ and $\sum_{j=1}^n (-1)^{n-j}{m-1-j\choose n-j}{m\choose j} = 1 - (-1)^n{m-1\choose n}$ which follows from Lemma~\ref{lemma:inclusion-exclusion}.
Since $\psi(\mathbb{S}^n)=0$ if $n$ is odd and $\psi(\mathbb{S}^n)=2$ if $n$ is even, \eqref{eq:centralsphererecoverWeibel} recovers the $k=0$ case of \cite[Theorem~1]{Weibel12}.

\bigskip

\paragraph{General case}
Next, we present an approach to handle the origin directly in $\mathbb{R}^{n+1}$, which will allow us to address the general case where the lifted Newton polytopes of the maxout units need not be full-dimensional. Our goal is to prove the following theorem:
\begin{theorem}[Number of regions of a simple central arrangement]
\label{thm:central-arbdims}
For a central simple arrangement $\A$ of $m$ maxout units of ranks $k_1,\ldots, k_m\geq 2$ in $\mathbb{R}^{n+1}$, $m\geq %
n+1$, %
\[ r(\A)=  {m-1 \choose n} + \sum_{j=0}^{n} (-1)^{n-j}{m-1-j\choose n-j} \sum_{S\in {[m]\choose j}} r(\A_S). \]
\end{theorem}

To prove Theorem~\ref{thm:central-arbdims}, we will use use Theorem~\ref{thm:posetcounting} to write $r(\mathcal{A})$ as a sum of terms $\psi(y)\mu_{L(\mathcal{A})}(y)$ over $y\in L(\mathcal{A})$. For the elements $y\in L(\mathcal{A})\setminus\{0\}$ we can use similar arguments as before. To handle $\{0\}=\hat 1\in L(\mathcal{A})$ we will use
the Cross-cut Theorem:
\begin{theorem}[Cross-cut Theorem {\cite{Rota}}]\label{thm:cross-cut}
  Let $L$ be a finite lattice. Let $X$ be a subset of $L$ such that $\hat 0\not\in X$ and such that if $y\in L$, $y\neq \hat 0$, then some $x\in X$ satisfies $x\leq y$. Let $N_k$ be the number of $k$ element subsets of $X$ with join $\hat 1$. Then $\mu_L(\hat 0,\hat 1) = N_0 -N_1 + N_2 \mp \cdots$.
\end{theorem}

In the following we will evaluate this formula for the case of the intersection poset $L(\mathcal{A})$ of a simple central arrangement, and split the expression into terms corresponding to sub-arrangements $\mathcal{A}_S$ with $0\leq |S|\leq n$.
We will use following definitions.

\begin{definition}
  Let $\A$ be a simple central arrangement of $m$ maxout units. For $S\subseteq[m]$ of cardinality $|S|\leq n$, let $N_k^S$ denote the number of $k$ element subsets of $\A_S$ with join $\{0\}$.
  For $k>n$, let $N_k^\ast$ denote the number of $k$ element subsets of $\A$ which contain atoms of at least $n+1$ units. %
  Note that their join is necessarily $\{0\}$ as $\A$ is simple.
\end{definition}

First we note that the terms involving more than $n$ units can be grouped as follows.
\begin{lemma}\label{lem:central-3}
  Let $\A$ be a simple central arrangement of $m$ maxout units in $\RR^{n+1}$, and let $M\coloneqq |\A|$ be the number of its atoms. Then $\sum_{k=n+1}^M (-1)^k N_k^\ast = - (-1)^n {m-1 \choose n}$.
\end{lemma}
\begin{proof}%
The statement follows from reformulations obtained by splitting $k$ element subsets of $\A$ with atoms from $\A_{\{i\}}$, $i\in S$, into a disjoint union of non-empty $J_i\subseteq \A_{\{i\}}$, $i\in S$: %
\begin{align*}
&\sum_{r=n+1}^M (-1)^r N_r^\ast
=  \sum_{r=n+1}^m \sum_{S\in{[m]\choose r}}
\sum_{ \substack{ \emptyset\neq J_i\subseteq \A_{\{i\}}\\ \text{for all }i\in S}}
\!\!(-1)^{\sum_{i\in S}|J_i|}
=  \sum_{r=n+1}^m \sum_{S\in{[m]\choose r}} \prod_{i\in S} %
\sum_{j_i=1}^{|\A_{\{i\}}|}{|\A_{\{i\}}|\choose j_i} (-1)^{j_i}%
\\
&=  \sum_{r=n+1}^m \sum_{S\in{[m]\choose r}} \prod_{i\in S} \left(-1\right)
=  \sum_{r=n+1}^m  {m \choose r} (-1)^{r}
=   -\sum_{r=0}^n  { m \choose r} (-1)^{r}
= - (-1)^n {m-1 \choose n}. \qedhere
\end{align*}
\end{proof}

Now using Theorem~\ref{thm:cross-cut} and Lemma~\ref{lem:central-3}, we obtain the following description of $\mu_{L(\mathcal{A})}(\{0\})$. %

\begin{lemma}
\label{lem:central}
Let $\A$ be a simple central arrangement of $m$ maxout units in $\RR^{n+1}$. Then
$$
\mu_{L(\A)}(\{0\})
= - (-1)^n {m-1 \choose n}
+ \sum_{j=0%
}^n (-1)^{n-j} {m-1-j\choose n-j} \sum_{S\in{[m]\choose j}} \mu_{L(\A_S)}(\{0\}) \mathds{1}_{\{0\}\in L(\A_S)}.
$$
Here $\mathds{1}_{\{0\}\in L(\A_S)}$ takes value $1$ if $\{0\}\in L(\A_S)$ and $0$ otherwise.
\end{lemma}

\begin{proof}%
Let $M\coloneqq |\A|$ be the number of atoms of $\A$, and for a subset $A\subseteq \A$ let $\Supp(A)\subseteq[m]$ denote the minimal $S\subseteq[m]$ with $A\subseteq \A_S$. With Theorem~\ref{thm:cross-cut} we can decompose $\mu_{L(\A)}(\{0\}) = \mu_{L(\A)}(\hat0,\hat1)$ as 
\allowdisplaybreaks
\begin{align*}%
    &\mu_{L(\A)}(\{0\})  \textoverset[O]{Thm.~\ref{thm:cross-cut}}{=} \sum_{k=0}^M (-1)^k N_k
    =\bigg[\sum_{k=n+1}^M (-1)^k N_k^\ast\bigg] + \bigg[\sum_{k=0}^M (-1)^k \sum_{\substack{A\in\binom{\mathcal A}{k}\\ \text{join}(A)=\{0\} \\ |\Supp(A)|\leq n}} 1\bigg] \\
    &\textoverset[O]{Lem.~\ref{lemma:inclusion-exclusion}}{=} \bigg[\sum_{k=n+1}^M (-1)^k N_k^\ast\bigg] + \bigg[\sum_{k=0}^M (-1)^k \!\!\!\sum_{\substack{A\in\binom{\mathcal A}{k}\\ \text{join}(A)=\{0\} \\ |\Supp(A)|\leq n}} \sum_{j=0}^n (-1)^{n-j}\binom{m-1-j}{n-j}\sum_{S\in\binom{[m]}{j}}\mathbb{1}_{\Supp(A)\subseteq S}\bigg] \\
    &\textoverset[O]{}{=} \bigg[\sum_{k=n+1}^M (-1)^k N_k^\ast\bigg] + \bigg[ \sum_{j=0}^n (-1)^{n-j}\binom{m-1-j}{n-j}\sum_{S\in\binom{[m]}{j}}\sum_{k=0}^M (-1)^k \!\!\!\sum_{\substack{A\in\binom{\mathcal A}{k}\\ \text{join}(A)=\{0\} \\ |\Supp(A)|\leq n}} \mathbb{1}_{\Supp(A)\subseteq S}\bigg] \\
    &\textoverset[O]{}{=} \bigg[\sum_{k=n+1}^M (-1)^k N_k^\ast\bigg] + \bigg[\sum_{j=0}^n (-1)^{n-j} {m-1-j\choose n-j} \sum_{S\in{[m]\choose j}}\sum_{k=0}^M (-1)^k N_k^S\bigg]\\
    &\textoverset[O]{}{=} \bigg[\sum_{k=n+1}^M (-1)^k N_k^\ast\bigg] + \bigg[\sum_{j=0}^n (-1)^{n-j} {m-1-j\choose n-j} \sum_{S\in{[m]\choose j}} \mu_{L(\A_S)}(\{0\}) \mathds{1}_{\{0\}\in L(\A_S)}\bigg]\\
    &\textoverset[O]{Lem.~\ref{lem:central-3}}{=} - (-1)^n {m-1 \choose n} + \bigg[\sum_{j=0}^n (-1)^{n-j} {m-1-j\choose n-j} \sum_{S\in{[m]\choose j}} \mu_{L(\A_S)}(\{0\}) \mathds{1}_{\{0\}\in L(\A_S)}\bigg]. \qedhere
\end{align*}
\end{proof}

We now have all the supporting results we need to prove Theorem~\ref{thm:central-arbdims}.

\begin{proof}[Proof of Theorem~\ref{thm:central-arbdims}]
Similar to the proof of Theorem~\ref{thm:facessimple}, note that if $\A$ is simple, then any element in $L(\mathcal{A})\setminus\{0\}$ can be written as an intersection of at most $n$ atoms. Moreover, any $y\in L(\mathcal{A})\setminus\{0\}$ is contained in $L(\mathcal{A}_S)$ if and only if its support is contained in $S$.
This allows us to apply Lemma~\ref{lemma:inclusion-exclusion} as in the proof of Theorem~\ref{thm:facessimple} for $y\in L(\mathcal{A})\setminus\{0\}$. %
Hence %
\allowdisplaybreaks
\begin{align*}\label{eq:central-1}
  &r(\mathcal{A}) \textoverset[O]{Thm.~\ref{thm:posetcounting}}{=} (-1)^{n+1} \psi(\{0\})\mu_{L(\mathcal{A})}(\{0\}) + (-1)^{n+1} \sum_{y\in L(\mathcal{A})\setminus\{0\}}\psi(y)\mu_{L(\mathcal{A})}(y)\nonumber \\
  &\textoverset[O]{Lem.~\ref{lemma:inclusion-exclusion}}{=} (-1)^{n+1} \mu_{L(\mathcal{A})}(\{0\}) + (-1)^{n+1} \sum_{j=0}^{n} (-1)^{n-j}{m-1-j\choose n-j} \!\!\!\sum_{S\in {[m]\choose j}}  \sum_{y\in L(\mathcal{A}_S)\setminus\{0\}}\!\!\!\!\! \psi(y)\mu_{L(\mathcal{A}_S)}(y) \\
  &\textoverset[O]{Lem.~\ref{lem:central}}{=} {m-1 \choose n} + (-1)^{n+1} \sum_{j=0}^{n} (-1)^{n-j}{m-1-j\choose n-j} \sum_{S\in {[m]\choose j}}  \Big(\mu_{L(\A_S)}(\{0\}) \mathds{1}_{\{0\}\in L(\A_S)} \\
  &\hspace{98mm} + \sum_{y\in L(\mathcal{A}_S)\setminus\{0\}} \psi(y)\mu_{L(\mathcal{A}_S)}(y) \Big) \\
  &\textoverset[O]{}{=} {m-1 \choose n} + (-1)^{n+1} \sum_{j=0}^{n} (-1)^{n-j}{m-1-j\choose n-j} \sum_{S\in {[m]\choose j}} \sum_{y\in L(\mathcal{A}_S)} \psi(y)\mu_{L(\mathcal{A}_S)}(y) \\
  &\textoverset[O]{Thm.~\ref{thm:posetcounting}}{=} {m-1 \choose n} + \sum_{j=0}^{n} (-1)^{n-j}{m-1-j\choose n-j} \sum_{S\in {[m]\choose j}} r(\A_S). \qedhere
\end{align*}
\end{proof}

Recall that units of rank $1$ have an empty arrangement and can be ignored.
Theorem~\ref{thm:central-arbdims} generalizes \eqref{eq:centralsphererecoverWeibel} and Weibel's result \cite[Theorem~1 for $k=0$]{Weibel12} by removing the requirement that
$k_1,\ldots, k_m\geq n+2$. %
We illustrate Theorem~\ref{thm:central-arbdims} in the next example.

\begin{example}
\label{ex:central-arbdims} \
\begin{enumerate}[leftmargin=*]
\item
Consider the arrangement of $m=2$ maxout units of ranks $k_1=3$, $k_2=2$ in $\mathbb{R}^d$, $d=n+1=2$ shown in the left part of Figure~\ref{fig:maxoutarrangementcentral}.
In this example $\mu(\mathbb{R}^2)=1$, $\mu(H^i_{ab})=-1$,  $\mu(\textcolor{red}{\bullet})=3$.
By Theorem~\ref{thm:posetcounting}, $r(\mathcal{A}) = \sum_{y\in L(\mathcal{A})} \psi(y)\mu_{L(\mathcal{A})}(y) = (-1)^2 1 + 0 (-1-1-1) + (-1)^1(-1) + (-1)^0 3 = 5$.
By Theorem~\ref{thm:central-arbdims}, this can be written as
$r(\mathcal{A}) = \psi(\mathbb{S}^n){m-1\choose n} + \sum_{j=1}^{n} (-1)^{n-j}{m-1-j\choose n-j} \sum_{S\in {[m]\choose j}} r(\mathcal{A}_S) = 0 + (-1)^0{0\choose 0}(3+2) = 5$.
\item
Consider now the arrangement of $m=3$ maxout units of ranks $k_1=3$, $k_2=3$, $k_3=2$ in $\mathbb{R}^d$, $d=n+1=3$, shown in the right part of Figure~\ref{fig:maxoutarrangementcentral}.
By Theorem~\ref{thm:central-arbdims}, the number of regions is $r(\A) = 2{3-1\choose 2} + (-1)^{2-1}{3-1-1\choose 2-1}(3+3+2) + (-1)^{2-2}{3-1-2\choose 2-2}(9 +6+6 ) = 15$.
\end{enumerate}
\end{example}

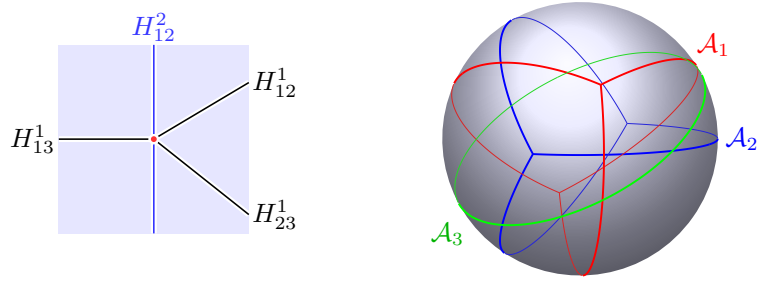
\begin{figure}
    \centering
\begin{tabularx}{10cm}{m{5cm}m{5cm}}
\begin{tikzpicture}[every node/.style={black,above right, inner sep=1pt}]

\path[fill=blue!10] (-1.25,-1.25) rectangle (1.25cm,1.25cm);

\draw[name path=line11, double=black, white, thick] (0,0) -- (1.25,.75) node [right] {$H^1_{12}$};
\draw[name path=line12, double=black, white, thick] (0,0) -- (1.25,-1) node [right] {$H^1_{23}$};
\draw[name path=line13, double=black, white, thick] (0,0) -- (-1.25,0) node [left] {$H^1_{13}$};

\draw[name path=line21, double=blue!80, white, thick] (0,-1.25) -- (0,1.25) node [above] {\textcolor{blue!80}{$H^2_{12}$}};

\fill[name intersections={of=line11 and line12,total=\t}, draw=white, thick] {(intersection-1) circle (1.5pt) node {}};

\foreach \i in {1}{
	\foreach \j in {1,2,3}{
       \fill[name intersections={of={line2\i} and {line1\j}, total=\t}, red!80, draw=white, thick][]
           \ifnum\t=0
           {};
           \else
        	\foreach \s in {1,...,\t}{(intersection-\s) circle (1.5pt) node {} } ;
         \fi
   	}
}
\end{tikzpicture}

&
\begin{tikzpicture}
\node at (0,0) {\includegraphics[width=4cm]{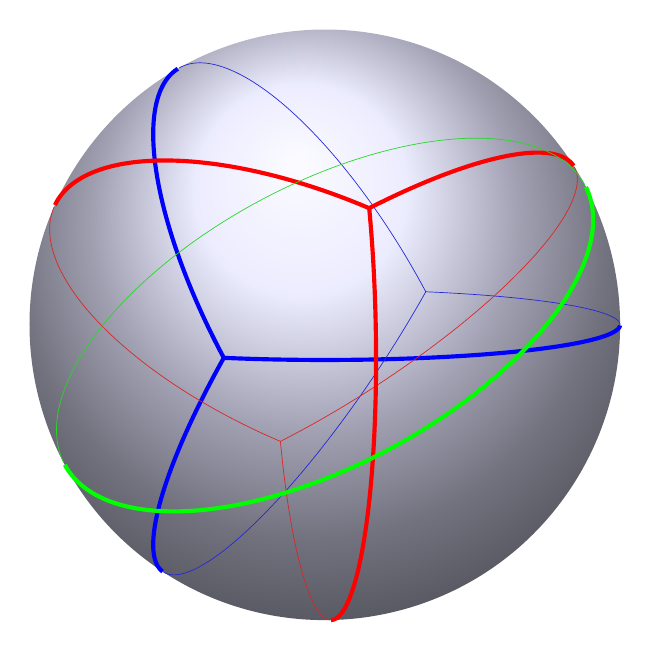}};
\node at (1.75,1.25) {\textcolor{red}{$\A_1$}};
\node at (2.125,0) {\textcolor{blue}{$\A_2$}};
\node at (-1.75,-1.25) {\textcolor{green!70!black}{$\A_3$}};
\end{tikzpicture}
\end{tabularx}
\vspace{-3mm}
\caption{%
Shown are two central maxout arrangements, one in $\mathbb{R}^2$ and one $\mathbb{R}^3$ (for clarity we show only the intersection with $\mathbb{S}^2$), discussed in Example~\ref{ex:central-arbdims}.
}
    \label{fig:maxoutarrangementcentral}
\end{figure}

\subsection{Strictly upper vertices or bounded regions}
We now have counting formulas for the number of upper vertices and the total number of vertices of a Minkowski sum of polytopes. Given a sharp upper bound on the total number of vertices, we can obtain a sharp upper bound on the number of upper vertices if we also have an appropriate lower bound on the number of strict lower vertices. In this subsection we derive such a lower bound.

The strict upper vertices of a Minkowski sum correspond
to the bounded regions that a central maxout arrangement defines on a hyperplane that does not intersect the origin, which are the regions that do not intersect the negated hyperplane.
These are also equivalent to the bounded regions of a non-central arrangement in one dimension lower.
The case of central hyperplane arrangements was studied in~\cite[Theorem~3.2]{Greene1983ONTI}, showing that, in that case, the induced arrangement over a general hyperplane has $\mu(\hat 0, \hat 1)$ relatively bounded regions.
A particular challenge in the case of maxout arrangements is that the atoms need not be symmetric. This means that the number of bounded regions it defines over a hyperplane may depend on the particular choice of the hyperplane. We obtain a lower bound that is independent of this choice.

\begin{theorem}[Lower bound on the number of bounded regions]\label{thm:lowerbound_strict_lower}
Let $\A$ be a simple central maxout arrangement in $\mathbb{R}^{n+1}$. %
Let $g=\{x\in \mathbb{R}^{n+1}\colon \langle x,w\rangle =1\}$ be an affine hyperplane that does not contain the origin.
Let $\A^g=\{g\cap H \neq \emptyset \colon H\in \A, H\not\supseteq g\}$.
Then the number of regions of $\A$ which do not intersect $g$,
i.e.\ regions consisting of points with $\langle x,w\rangle\leq 0$,
satisfies
\begin{align*}
r(\A) - r(\A^g)
\geq & {m-1\choose n}.
\end{align*}
\end{theorem}

\begin{proof}
By Theorem~\ref{thm:central-arbdims} and Theorem~\ref{thm:facessimple}, we have
$$
r(\A) - r(\A^g) = {m-1\choose n} + \sum_{j=0}^n(-1)^{n-j}{m-1-j\choose n-j}\sum_{S\in{[m]\choose j}}(r(\A_S) - r(\A_S^g) ).
$$
It remains to show that the second summand is non-negative.
We define the support of a region $R$ of $\A_S$ resp.\ $\A^g_S$ to be the minimal $Q\subseteq [m]$ such that $R$ is a region of $\A_Q$ resp.\ $\A_Q^g$.
Let $s(\A_Q)$ resp.\ $s(\A_Q^g)$ be the number of regions of $\A_Q$ resp.\ $\A^g_Q$ with support $Q$, so that $r(\A_S)=\sum_{Q\subseteq S} s(\A_Q)$ and $r(\A_S^{g})=\sum_{Q\subseteq S} s(\A_Q^{g})$. Then
\allowdisplaybreaks
\begin{align*}
  & \sum_{j=0}^n(-1)^{n-j}{m-1-j\choose n-j}\sum_{S\in{[m]\choose j}}(r(\A_S) - r(\A_S^g)) \\
  & \textoverset[O]{}{=} \sum_{j=0}^n(-1)^{n-j}{m-1-j\choose n-j}\sum_{S\in{[m]\choose j}} \sum_{Q\subseteq S} (s(\A_Q) - s(\A_Q^g)) \\
  & \textoverset[O]{}{=} \sum_{j=0}^n(-1)^{n-j}{m-1-j\choose n-j}\sum_{\substack{Q\subseteq [m]\\ |Q|\leq j}} \binom{m-|Q|}{j-|Q|} (s(\A_Q) - s(\A_Q^g)) \\
  & \textoverset[O]{}{=} \sum_{\substack{Q\subseteq [m]\\ |Q|\leq n}} (s(\A_Q) - s(\A_Q^g)) \sum_{j=0}^n(-1)^{n-j}{m-1-j\choose n-j} \binom{m-|Q|}{j-|Q|} \\
  & \textoverset[O]{Lem.~\ref{lemma:inclusion-exclusion}}{=} \sum_{\substack{Q\subseteq [m]\\ |Q|\leq n}} (s(\A_Q) - s(\A_Q^g)).
\end{align*}
Finally, note that each summand in the final expression is non-negative since the regions of a sub-arrangement $\A_Q$ intersecting $g$ is a subset of the regions of %
$\A_Q$, and the same holds true for regions that are not contained in a smaller arrangement $\A_{Q'}$, $Q'\subsetneq Q$.
\end{proof}

\section{Discussion and outlook}
\label{sec:discussion}\vspace{-2mm}

We presented sharp explicit upper bounds for the number of linear regions of the functions that can be represented by shallow maxout networks with and without biases.
These results can be regarded as upper bound theorems for tropical arrangements or upper bound theorems for the number of vertices and upper vertices of Minkowski sums of polytopes with given numbers of vertices.
As a direct application of our sharp bounds for shallow maxout networks, we obtained asymptotically tight bounds for deep maxout networks.
These results substantially improve previous lower and upper bounds.

We presented counting formulas for the number of faces of maxout arrangements in terms of the intersection poset. In the case of simple arrangements or Minkowski sums of polytopes in general orientations, we obtained formulas %
in terms of sub-arrangements or Minkowski subsums.
We also presented a lower bound on the number of strict lower vertices of Minkowski sums of polytopes in general orientations, which correspond to the bounded regions of maxout arrangements.
Our discussion connects the theoretical analysis of artificial neural networks, tropical geometry, and geometric combinatorics.
The results that we have presented can serve as the basis for addressing several other problems:
\begin{itemize}[leftmargin=*]
    \item
One possible extension of the results we presented here are explicit formulas for the maximum number of faces of any dimensions or also for the number of bounded faces. Explicit formulas for lower-dimensional faces are of particular importance for the combinatorial complexity of tropical varieties~\cite{Joswig_2017}, which are intersections of tropical hypersurfaces, and consequently also for the complexity of many algorithms in tropical geometry.
\item
Further refining the bounds for deep networks is also an interesting endeavor for future work. Even the case of ReLU networks is still the subject of intense investigation.
Another interesting avenue is the explicit number of faces for specific families of non-simple arrangements, for example those that one might obtain in convolutional networks or graph and simplicial networks, which have been recently studied in the ReLU case \cite{xiong2020number,bodnar2021weisfeiler}.
\item
An interesting open problem is the estimation of the expected number of faces for a given probability distribution over the parameters of shallow and deep maxout networks. The case of ReLU networks was recently studied in \cite{pmlr-v97-hanin19a,NIPS2019_8328}.
In shallow ReLU networks, any generic parameter gives rise to the maximum number of regions (the number of regions of a generic hyperplane arrangement or equivalently the number of upper vertices of a zonotope).
In contrast, for shallow maxout networks, generic choices of parameters can result in different numbers of regions.
The expected number of regions of a single maxout unit with Gaussian weights corresponds to the number of (upper) vertices of a Gaussian polytope, which has been studied in the literature \cite{HMR04}. However, for a maxout layer one would need to consider Minkowski sums of random polytopes, which to our knowledge have not yet been studied.
\item
A further aspect of interest is the development of parameter initialization strategies that would allow for faster optimization or better algorithmic biases when training maxout networks.
We have presented ways to select parameters that lead to the maximum number of regions for shallow maxout networks and to asymptotically maximal number of regions for deep maxout networks.
Other properties of the initialization that can be considered include the normalization of the activation values across layers and the distribution of linear regions over the space of inputs.
Related aspects for the case of ReLU networks have been studied in \cite{he2015delving,NEURIPS2018_d81f9c1b,Steinwart2019ASL,86441,Zhang2020Empirical}.
\end{itemize}
\vspace{-4mm}

\subsection*{Acknowledgment}
We are grateful to both Raman Sanyal and especially Karim Adiprasito for discussing their Upper Bound Theorem for Minkowski Sums with us.
Guido Mont\'ufar has been supported by the European Research Council (ERC) under the European Union’s Horizon 2020 research and innovation programme (grant n\textsuperscript{o}~757983).
Yue Ren %
has been supported by UK Research and Innovation (UKRI)
under the Future Leaders Fellowship programme (grant n\textsuperscript{o}MR/S034463/1).
\vspace{-2mm}

\bibliographystyle{plain}
\bibliography{referenzen}

\end{document}